\documentclass[11pt,leqno,a4paper]{article}
\usepackage{amsthm}
\usepackage{indentfirst}
\usepackage{amsmath}
\usepackage{mathrsfs}
\usepackage{txfonts}
\usepackage{amssymb,enumerate}
\usepackage{hyperref}
\usepackage{titlesec}
\usepackage{url}
\usepackage{rotating}
\usepackage{booktabs}
\usepackage{extarrows}

\usepackage{mathtools}

\renewcommand\thesubsection{\thesection.\Alph{subsection}}
\newcommand{\sbnal}{0.3em}   
\titleformat{\subsection}[runin]{\bfseries}{\thesubsection}{\sbnal}{}[]

\newtheorem{mainthm}{Theorem}
\newtheorem*{conj}{Conjecture}

\newtheorem{thm}{Theorem}[section]
\newtheorem{prop}[thm]{Proposition}

\newtheorem{lem}[thm]{Lemma}
\newtheorem{cor}[thm]{Corollary}
\newtheorem{rem}[thm]{Remark}

\newcommand{\GL}{\operatorname{GL}}

\newcommand{\PSL}{\operatorname{PSL}}
\newcommand{\PSU}{\operatorname{PSU}}
\newcommand{\CSp}{\operatorname{CSp}}
\newcommand{\Sp}{\operatorname{Sp}}
\newcommand{\PSp}{\operatorname{PSp}}
\newcommand{\SO}{\operatorname{SO}}
\newcommand{\I}{\operatorname{I}}
\newcommand{\J}{\operatorname{J}}
\newcommand{\D}{\operatorname{D}}
\newcommand{\Aut}{\operatorname{Aut}}

\newcommand{\Out}{\operatorname{Out}}
\newcommand{\Ind}{\operatorname{Ind}}
\newcommand{\Res}{\operatorname{Res}}
\newcommand{\Irr}{\operatorname{Irr}}
\newcommand{\IBr}{\operatorname{IBr}}
\newcommand{\Bl}{\operatorname{Bl}}
\newcommand{\Rad}{\operatorname{Rad}}
\newcommand{\dz}{\operatorname{dz}}
\newcommand{\Cl}{\operatorname{Cl}}
\newcommand{\Alp}{\operatorname{Alp}}
\newcommand{\diag}{\operatorname{diag}}

\newcommand{\bl}{\operatorname{bl}}
\newcommand{\rO}{\operatorname{O}}
\newcommand{\Ker}{\operatorname{Ker}}
\newcommand{\rk}{\operatorname{rk}}

\newcommand{\group}[1]{\langle#1\rangle}
\newcommand{\Group}[1]{\left\langle#1\right\rangle}
\newcommand{\set}[1]{\{#1\}}
\newcommand{\Set}[1]{\left\{#1\right\}}

\newcommand{\simi}{\!\!\sim\!\!}

\newcommand{\F}{\mathbb{F}}
\newcommand{\barF}{\overline{\mathbb{F}}}
\newcommand{\Z}{\mathbb{Z}}

\newcommand{\bG}{\mathbf{G}}
\newcommand{\tG}{\tilde{G}}
\newcommand{\tbG}{\tilde{\mathbf{G}}}
\newcommand{\hG}{\hat{G}}
\newcommand{\tbT}{\tilde{\mathbf{T}}}
\newcommand{\bV}{\mathbf{V}}

\newcommand{\tB}{\tilde{B}}
\newcommand{\tb}{\tilde{b}}
\newcommand{\cB}{\mathcal{B}}
\newcommand{\tcB}{\tilde{\mathcal{B}}}
\newcommand{\tD}{\tilde{D}}
\newcommand{\tR}{\tilde{R}}
\newcommand{\tC}{\tilde{C}}
\newcommand{\hC}{\hat{C}}
\newcommand{\tN}{\tilde{N}}
\newcommand{\hN}{\hat{N}}

\newcommand{\zero}{\mathbf{0}}
\newcommand{\bc}{\mathbf{c}}

\newcommand{\cF}{\mathcal{F}}
\newcommand{\cE}{\mathcal{E}}
\newcommand{\cR}{\mathcal{R}}
\newcommand{\cW}{\mathcal{W}}
\newcommand{\cK}{\mathcal{K}}
\newcommand{\cJ}{\mathcal{J}}
\newcommand{\fZ}{\mathfrak{Z}}
\newcommand{\sC}{\mathscr{C}}
\newcommand{\fS}{\mathfrak{S}}

\newcommand{\hE}{\hat{E}}
\newcommand{\hF}{\hat{F}}

\newcommand{\ts}{\tilde{s}}
\newcommand{\tz}{\tilde{z}}

\newcommand{\ttheta}{\tilde{\theta}}
\newcommand{\tvarphi}{\tilde{\varphi}}

\newcommand{\tchi}{\tilde{\chi}}
\newcommand{\tpsi}{\tilde{\psi}}
\newcommand{\tOmega}{\tilde{\Omega}}

\newcommand{\htheta}{\hat{\theta}}
\newcommand{\hpsi}{\hat{\psi}}

\begin{document}

\title{The inductive blockwise Alperin weight condition for $\PSp_{2n}(q)$ and odd primes
\footnote{Supported by National Natural Science Foundation of China (No. 11901478 and No. 11631001) and Fundamental Research Funds for the Central Universities (No. 2682019CX48).}}

\author{Conghui Li\\ \\
\small{Department of Mathematics, Southwest Jiaotong University,}\\
\small{Chengdu 611756, China.}\\
\small{Email:liconghui@swjtu.edu.cn}}

\maketitle

\begin{abstract}
In this paper, using a criterion given by J. Brough and B. Sp\"ath recently, we verify the inductive blockwise Alperin weight condition for the simple groups $\PSp_{2n}(q)$ and any odd prime $\ell$ not dividing $q$ under some assumptions concerning the decomposition matrices.
\end{abstract}


\section{Introduction}\label{sect:intro}

An important longstanding global-local conjecture, called the blockwise Alperin weight conjecture, is proposed by Alperin in \cite{Al87}.
The conjecture is stated as follows; see \S\ref{subsec:general} for the definition of weight.

\begin{conj}[Alperin]
Let $G$ be a finite group, $\ell$ a prime and $B$ an $\ell$-block of $G$.
Denote the set of all $G$-conjugacy classes of $B$-weights by $\Alp(B)$, then $|\Alp(B)|=|\IBr(B)|$.
\end{conj}

This conjecture has been verified for numerous groups, such as symmetric groups, many sporadic simple groups and many finite groups of Lie type.
Here we just mention the papers of Alperin-Fong \cite{AF90} and An \cite{An94}, which classify the weights of some classical groups for odd primes.

\paragraph{}
Since the general proof of this conjecture seems extremely difficult, an accessible way is to reduce it to simple groups.
First, Navarro and Tiep reduced the non-blockwise version of this conjecture to simple groups in \cite{NT11}, and then Sp\"ath reduced the blockswise version to simple groups in \cite{Sp13}.
Thus to prove this conjecture, it suffices to verify the inductive blockwise Alperin weight (BAW) condition defined in \cite{Sp13} for all simple groups.
This inductive condition has been verified for many cases.
In \cite{Sp13}, Sp\"ath gives a verification for finite simple groups of Lie type and the defining characteristic.
Malle verifies it in \cite{Ma14} for simple alternating groups, Suzuki groups and Ree groups.
All blocks with cyclic defect groups have been proved to satisfy this inductive condition by Koshitani and Sp\"ath in \cite{KS16a,KS16b}.
Schulte gives a verification for it for simple groups of type $G_2$ and $^3D_4$ in \cite{Sch16}.
A special case of type A is verified by Li-Zhang in \cite{LZ18,LZ19}.
Feng verifies it for the unipotent blocks of type A in \cite{Feng19}, and together with Z. Li and J. Zhang for unipotent blocks of classical groups and some other cases of classical groups in \cite{FLZ19}.
Some more particular simple groups of small rank are considered in \cite{SF14,FLL17a,BSF19,LL19}.

\paragraph{}
The inductive BAW condition is highly complicated, which makes it very challenging to verify it for simple groups of Lie type.
Roughly speaking, it consists of two parts: the first requires an equivariant bijection between irreducible Brauer characters and weights; the second is some requirements concerning the Clifford theory for characters under this bijection.
Fortunately, Brough and Sp\"ath prove a criterion for this inductive condition for simple groups with abelian outer automorphism groups, which makes it possible to consider the simple groups of Lie type B and C; see Theorem \ref{thm:criterion} for the statement.
We consider the simple group $\PSp_{2n}(q)$ and odd primes $\ell\nmid q$ in this paper.
First, assume $q$ is odd, in which case, according to the criterion, the main task is to establish the BAW conjecture itself for the conformal symplectic groups.

\begin{mainthm}\label{mainthm-1}
Let $p$ be an odd prime, $q=p^f$ and $\ell$ an odd prime different from $p$, then the blockwise Alperin weight conjecture holds for $\CSp_{2n}(q)$ and $\ell$.
\end{mainthm}

The proof of the above theorem relies heavily on the results of Fong-Srinivasan in \cite{FS89} and An in \cite{An94}.
To continue the verification, we prove that a certain bijection between irreducible Brauer characters and weights of conformal symplectic groups satisfies the requirements of part (iii) in Theorem \ref{thm:criterion}.
The remaining requirements concerning the symplectic groups can be verified easily, using a result of Cabanes and Sp\"ath in \cite{CS17C} and a result in \cite{Li19} of the author.
In the process, to obtain information about irreducible Brauer characters, we need to make some assumptions about the decomposition matrices.
Thus the main result for $q$ odd is stated as follows.

\begin{mainthm}\label{mainthm-2}
Keep the assumptions in Theorem \ref{mainthm-1}.
Assume the decomposition matrices of $\Sp_{2n}(q)$ and $\CSp_{2n}(q)$ with respect to $\cE(\Sp_{2n}(q),\ell')$ and $\cE(\CSp_{2n}(q),\ell')$ are unitriangular.
Then the inductive blockwise Alperin weight condition holds for the simple group $\PSp_{2n}(q)$ and $\ell$.
\end{mainthm}

Here, recall that $\cE(G,\ell') = \bigcup_{s\in G^*_{ss,\ell'}} \cE(G,(s))$ for a finite group $G$ of Lie type.

\paragraph{}
Next, we consider the relatively easy case for $q=2^f$.
\begin{mainthm}\label{mainthm-3}
Let $q=2^f$ and $\ell$ be an odd prime, then the blockwise Alperin weight conjecture holds for $\Sp_{2n}(2^f)$ and $\ell$.
Furthermore, assume $n\geq2$ and $(n,f)\neq(2,1)$ or $(3,1)$, then if the decomposition matrix of $\Sp_{2n}(2^f)$ to $\cE(\Sp_{2n}(2^f),\ell')$ is unitriangular, the inductive blockwise Alperin weight condition holds for $\Sp_{2n}(2^f)$ and $\ell$.
\end{mainthm}

The structure of this paper is as follows.
In \S\ref{sec:pre}, we give some notation and preliminaries.
Then in sections \S\ref{sec:prepare}, \S\ref{sec:weights}, \S\ref{sec:inductive}, we consider the case when $q$ is odd.
First, we consider the Brauer pairs associated to radical subgroups and recall the results about blocks of conformal symplectic groups in \S\ref{sec:prepare}.
In \S\ref{sec:weights}, we classify the weights of conformal symplectic groups and prove Theorem \ref{mainthm-1}.
In \S\ref{sec:inductive}, we verify the inductive BAW condition for the cases in Theorem \ref{mainthm-2}.
Finally, in \S\ref{sec:main-3}, we consider the case when $q$ is even.


\section{Preliminaries}\label{sec:pre}

\subsection{}\label{subsec:general}
For an element $g$ of a finite group $G$, we denote by $\Cl_G(g)$ its $G$-conjugacy class.
The notation for representations of finite groups in this paper is standard,
which can be found for example in \cite{NT89} except that
we use $\Ind$ and $\Res$ for induction and restriction and
use $\chi^0$ for the restriction of an ordinary character $\chi$ of $G$ to the $\ell$-regular elements.
We will consider the modular representations with respect to a fixed prime $\ell$, thus we will abbreviate $\ell$-Brauer characters, $\ell$-block, etc. as Brauer characters, blocks, etc.

For a finite group $G$, we denote by $\dz(G)$ the set of all irreducible defect zero characters of $G$.
For $K \unlhd G$ and $\theta\in\Irr(K)$, we set $\dz(G \mid \theta) = \dz(G) \cap \Irr(G \mid \theta)$.
For an $\ell$-subgroup $R$ of $G$ and a block $B$ of $G$, $\dz(N_G(R)/R,B)$ denotes the set of all characters $\varphi$ of $N_G(R)$ which lift characters in $\dz(N_G(R)/R)$ and satisfy $\bl(\varphi)^G=B$.
Here, $\bl(\varphi)$ is the block of $N_G(R)$ containing $\varphi$.
Here and often in the sequel, we will identify characters of $N_G(R)/R$ and their lifts to $N_G(R)$.

An $\ell$-weight of $G$ means a pair $(R,\varphi)$ with $\varphi\in\dz(N_G(R)/R)$.
In this case, $R$ is necessarily an $\ell$-radical subgroup of $G$, \emph{i.e.} $R=\rO_\ell(N_G(R))$, and $\varphi$ is called a weight character.
If furthermore $\varphi\in\dz(N_G(R)/R,B)$, then $(R,\varphi)$ is called a $B$-weight.
Denote the set of all $B$-weights by $\Alp^0(B)$ and the set of all $G$-conjugacy classes of $B$-weights by $\Alp(B)$.
For $(R,\varphi) \in \Alp^0(B)$, denote by $\overline{(R,\varphi)}$ the $G$-conjugacy class of $(R,\varphi)$.

Recently, J. Brough and B. Sp\"ath gave a new criterion for the inductive condition particularly suitable for simple groups of type B and C.
Denote by $\Rad(G)$ the set of all radical subgroups of $G$ and by $\Rad(G)/\simi G$ a $G$-transversal of $\Rad(G)$.

\begin{thm}[{\cite[Theorem]{BS19}}]\label{thm:criterion}
Let $S$ be a finite non-abelian simple group and $\ell$ a prime dividing $|S|$.
Let $G$ be an $\ell'$-covering group of $S$ and assume there are groups $\tG$, $E$ such that $G \unlhd \tG \times E$.
Assume $\cB\subseteq\Bl(G)$ is a $\tG$-stable set satisfying $(\tG E)_B\leq(\tG E)_\cB$ for any $B\in\cB$ and the following hold:
\begin{enumerate}[(i)]\setlength{\itemsep}{0pt}
\item \begin{itemize}\setlength{\itemsep}{0pt}
	 \item $G=[\tG,\tG]$ and $E$ is abelian,
	 \item $C_{\tG E}(G)=Z(\tG)$ and $\tG E/Z(\tG) \cong \Aut(G)$ by the natural map,
	 \item any element of $\IBr(\cB)$ extends to its stabilizer in $\tG$,
	 \item for any $R \in \Rad(G)$, any element of $\dz(N_G(R)/R,B)$ with $B\in\cB$ extends to its stabilizer in $N_{\tG}(R)/R$.
	 \end{itemize}
\item Let $\tcB=\Bl(\tG\mid\cB)$.
	There exists an $\IBr(\tG/G) \rtimes E_{\cB}$-equivariant bijection
	\[\tOmega_{\tcB}:\quad \IBr(\tcB) \to \Alp(\tcB)\]
	with $\tOmega_{\tcB}(\IBr(\tB)) = \Alp(\tB)$ for every $\tB \in \tcB$
	and $\J_G(\tpsi) = \J_G(\tOmega_{\tcB}(\tpsi))$ for every $\tpsi \in \IBr(\tcB)$.
\item For every $\tchi\in\IBr(\tG)$, there exists some $\chi_0\in\IBr(G\mid\tchi)$ such that
	\begin{itemize}\setlength{\itemsep}{0pt}
	\item $(\tG\rtimes E)_{\chi_0}=\tG_{\chi_0}\rtimes E_{\chi_0}$,
	\item $\chi_0$ extends to $G\rtimes E_{\chi_0}$.
	\end{itemize}
\item For every $\overline{(R,\psi_0)} \in \Alp(\cB)$, there exists some $x \in \tG$ with
	\begin{itemize}\setlength{\itemsep}{0pt}
	\item $(\tG E)_{R,\psi_0}= \tG_{R,\psi_0} (GE^x)_{R,\psi_0}$,
	\item $\psi_0$ extends to $(G\rtimes E^x)_{R,\psi_0}$.
	\end{itemize}
\end{enumerate}
Then the inductive blockwise Alperin weight condition holds for any $B$ in $\cB$ with abelian $\Out(G)_{\tG\textrm{-orbit of }B}$.
\end{thm}

For the definition of $\J_G(\tpsi)$ and $\J_G(\tOmega(\tpsi))$, see \cite{BS19}.
In the case when $G=\Sp_{2n}(q)$, $\tG=\CSp_{2n}(q)$ and $\ell$ is an odd prime, $\J_G(\tpsi) = \tG = \J_G(\tOmega(\tpsi))$ is trivially satisfied.

\subsection{}\label{sect:gLt}
Since the case $\PSp_2(q)=\PSL_2(q)$ has been considered, we may assume $n\geq2$.
We will frequently use the notation in \cite{FS89} and \cite{Li19}, so we review those which will be used from now on;
for more notation, we will refer to the definitions in \cite{FS89} and \cite{Li19} when they appear for the first time in the sequel.

Assume $q$ is odd.
Let $\bV$ ($\bV^*$) be the symplectic space of dimension $2n$ (the orthogonal space of dimension $2n+1$)
over the algebraic closure $\barF_p$ of the finite field $\F_p$ of $p$ elements
and $\bG=\Sp(\bV),\tbG=\CSp(\bV),\bG^*=\SO(\bV),\tbG^*=\D_0(\bV^*)$.
Denote by $F$ both the Frobenius maps on $\tbG$ and $\tbG^*$ defining an $\F_q$-structure and by $G,\tG,G^*,\tG^*$ the corresponding groups of fixed points.
These finite groups can be viewed as groups on a symplectic space $V$ or an orthogonal space $V^*$ over $\F_q$.
In particular, $G^* \cong \SO_{2n+1}(q)$.
The regular embedding $i: \bG \to \tbG$, its dual $i^*: \tbG^* \to \bG^*$ and related constructions are as in \cite[\S2.B]{Li19}.
Let $E=\group{F_p}$ be the group of field automorphisms,
then $\tG \rtimes E$ affords all automorphisms of the finite simple group $S=G/Z(G)$ when $n\geq2$.

When $q$ is a power of $2$ and $n\geq2$,
$G$ is simple if $(n,f)\neq(2,1)$ and is its own universal covering group of itself if furthermore $(n,f)\neq(3,1)$,
thus we will not need to introduce $\tbG,\tbG^*,\tG,\tG^*$.
In this case, there are isomorphisms of abstract groups $\bG \cong \bG^*$ (but not as algebraic groups) and $G \cong G^*$.

Assume $q=p^f$ is a power of an odd prime $p$ and $\ell$ is an odd prime different from $p$.
Let $\F_q$ be the field of $q$ elements.
The sets $\cF_0,\cF_1,\cF_2$ of polynomials are defined as in \cite[(1.7)]{FS89} and $\cF=\cF_0\cup\cF_1\cup\cF_2$.
For any $\Gamma\in\cF$, $\delta_\Gamma$ and $\varepsilon_\Gamma$ are defined as in \cite[(1.8),(1.9)]{FS89}.
We will use $\ts$ to denote a semisimple element of $\tG^*$ and use $s=i^*(\ts)$ to denote the image of $\ts$ under $i^*$, which is different from the convention in \cite{FS89}.
The $G^*$-conjugacy class of $s$ is determined by the multiplicity function $\Gamma\in\cF \mapsto m_\Gamma(s)$ and the type function $\Gamma\in\cF \mapsto \eta_\Gamma(s)$; see \cite[pp.125--126]{FS89}.

If $q=2^f$, we define $\cF_0$ as
\[\cF_0=\set{X-1}\]
and let $\cF_1,\cF_2$ be defined in the same way as in \cite[(1.7)]{FS89} with odd $q$ replaced by $2^f$.
Set then $\cF=\cF_0\cup\cF_1\cup\cF_2$.
For any $\Gamma\in\cF$, $\delta_\Gamma$ and $\varepsilon_\Gamma$ are defined in the same way as in \cite[(1.8),(1.9)]{FS89}.
The $G$-conjugacy class of a semisimple element $s$ is determined by the multiplicity function $\Gamma\in\cF \mapsto m_\Gamma(s)$; for the structure of $C_G(s)$, see \cite[Lemma 2.2]{SF14}.
Since $G^* \cong G$, this also gives a parametrization of conjugacy classes of semisimple elements of $G^*$ and a description of their centralizers.

\subsection{}\label{subsec:radical}
The radical subgroups of $G$ are given by An in \cite{An94}.
We recall the construction using the twisted basic subgroups introduced in \cite{Li19}.

Recall that $e$ is the multiplicative order of $q^2$ in $\Z/\ell\Z$ and $\ell$ is said to be linear or unitary if $\ell$ divides $q^2-1$ or $q^e+1$.
Set $\varepsilon=1$ or $-1$ when $\ell$ is linear or unitary.
Let $a=v(q^{2e}-1)$, where $v_\ell$ is the discrete valuation such that $v(\ell)=1$.
A group of symplectic type is a central product $Z_\alpha E_\gamma$ of a cyclic groups $Z_\alpha$ of order $\ell^{a+\alpha}$ and an extraspecial group $E_\gamma$ of order $\ell^{2\gamma+1}$; since only extraspecial groups of exponent $\ell$ occurs in radical subgroups, we always assume the exponent is $\ell$.
Let $R_{m,\alpha,\gamma}^0$ be the embedding of $Z_\alpha E_\gamma$ in $G_{m,\alpha,\gamma}^0:=\GL(m\ell^\gamma,\varepsilon q^{e\ell^\alpha})$ given in \cite[\S6.A]{Li19}.
Let $\bG_{m,\alpha,\gamma}=\Sp(2me\ell^{\alpha+\gamma},\barF_p)$ and $v_{m,\alpha,\gamma}$ be as in \cite[\S6.A]{Li19}.
Set $G_{m,\alpha,\gamma}^{tw}:=\bG_{m,\alpha,\gamma}^{v_{m,\alpha,\gamma}F}$, then the hyperbolic embedding $\hbar: G_{m,\alpha,\gamma}^0 \to G_{m,\alpha,\gamma}^{tw}$ is defined as in \cite[\S6.A]{Li19}.
The image of $R_{m,\alpha,\gamma}^0$ under $\hbar$ is denoted by $R_{m,\alpha,\gamma}^{tw}$.
Let $g_{m,\alpha,\gamma}$ be an element in $\bG_{m,\alpha,\gamma}$ such that $g_{m,\alpha,\gamma}^{-1}F(g_{m,\alpha,\gamma}^{-1})=v_{m,\alpha,\gamma}$; see for example \cite[Theorem 7.1]{CE04}.
The map $\iota: G_{m,\alpha,\gamma}^{tw} \to G_{m,\alpha,\gamma} := \bG_{m,\alpha,\gamma}^F$ defined by the conjugation by $g_{m,\alpha,\gamma}$ is an isomorphism.
Denote $R_{m,\alpha,\gamma}=\iota(R_{m,\alpha,\gamma}^{tw})$.

Let $\bc$ be defined as in \cite[\S6.A]{Li19}, then $R_{m,\alpha,\gamma,\bc}^{tw}=R_{m,\alpha,\gamma}^{tw}\wr A_\bc$ is called a twisted basic subgroup, which is a subgroup of $G_{m,\alpha,\gamma,\bc}^{tw}:=\bG_{m,\alpha,\gamma,\bc}^{v_{m,\alpha,\gamma,\bc}F}$, where $v_{m,\alpha,\gamma,\bc}$ and $\bG_{m,\alpha,\gamma,\bc}$ is as in \cite[\S6.A]{Li19}.
Let $g_{m,\alpha,\gamma,\bc}=g_{m,\alpha,\gamma}\otimes I_{\ell^{|\bc|}}$ and the isomorphism (and all similar homomorphisms) induced by conjugation by $g_{m,\alpha,\gamma,\bc}$ is again denoted by $\iota$.
Set $R_{m,\alpha,\gamma,\bc}=\iota(R_{m,\alpha,\gamma,\bc}^{tw})$, called a basic subgroup, which is conjugate to the basic subgroup denoted by the same symbol in \cite{An94}.
Obviously, $R_{m,\alpha,\gamma,\bc}=R_{m,\alpha,\gamma}\wr A_{\bc}$.

Let $\tau_{m,\alpha,\gamma,\bc}^{tw}$ be as in \cite[(6.3)]{Li19}.
The notation and results on centralizers and normalizers of the basic subgroups are collected in \cite[Lemma 6.4, Lemma 6.6]{Li19}.
Set $\tau_{m,\alpha,\gamma,\bc}=\iota(\tau_{m,\alpha,\gamma,\bc}^{tw})$.

\begin{rem}
We remark that all the above constructions also apply to the case when $q$ is even.
In fact, since $-1=1$ for even $q$, the structure of normalizers in \cite[Lemma 6.6]{Li19} can be even slightly simplified.
To be more specific, $N_{m,\alpha,\gamma}^{tw} = \hbar(N_{m,\alpha,\gamma}^0) \rtimes V_{m,\alpha,\gamma}^{tw}$ while for odd $q$, $\hbar(N_{m,\alpha,\gamma}^0) \cap V_{m,\alpha,\gamma}^{tw} = Z(G_{m,\alpha,\gamma}^{tw})$.
\end{rem}

Assume $q$ is odd.
By \cite{An94}, all radical subgroups of $G$ are conjugate to subgroups of the form $R_0 \times R_1 \times \cdots \times R_u$, where $R_0$ is the trivial group and $R_i=R_{m_i,\alpha_i,\gamma_i,\bc_i}$ is a basic subgroup for $i>0$.
In fact, this also holds for even $q$; see \S\ref{sec:main-3}.
For such a radical subgroup $R$ of $G$, let $v,g,\iota$ be as in \cite[Lemma 6.7]{Li19} and set $\tG^{tw}=\tbG^{vF}$.
In the sequel, when we say ``consider the twisted groups'' or ``transfer to twisted groups'', we will always mean transfer the problems to $\tG^{tw}$ using $\iota$.
The method to use twisted groups to consider local structures is introduced in \cite{CS17}.
This method has some technical advantages.
For example, the diagonal automorphisms and field automorphisms fix the twisted version of radical subgroups (see \cite[Lemma 6.8]{Li19}), which makes the calculation of automorphisms on weights easier.
See \S\ref{sec:inductive} for further applications of the twisted version of radical subgroups.

\subsection{}\label{subsec:ano-conju}
The version of basic subgroups given above is convenient when one considers the action of automorphisms on weights (see \cite{Li19}) and the extension problem.
Now, we will give another conjugate of the basic subgroup $R_{m,\alpha,\gamma,\bc}$, which is convenient when one considers the inclusion of some Brauer pairs in \S\ref{subsec:Brauerpair}.

Note that $R_{m,\alpha}$ is a cyclic subgroup of $\Sp_{2me\ell^\alpha}(q)$ of order $\ell^{a+\alpha}$ consisting of the elements of the form $z_{m,\alpha}(\zeta):=(\iota\circ\hbar)(\zeta I_m)$ for some $\zeta\in\barF_p$ with $\zeta^{\ell^{a+\alpha}}=1$.
Let $\zeta_\ell$ be a primitive $\ell$-th root of unity, then $R_{m,\alpha,\gamma}$ defined as above is conjugate to the group $R_{m,\alpha,\gamma,1}$ generated by the following elements
\[R_{m,\alpha}\otimes I_{\ell^\gamma}, \diag\left\{z_{m,\alpha}(1),z_{m,\alpha}(\zeta_\ell),\cdots,z_{m,\alpha}(\zeta_\ell^{\ell-1})\right\}, I_{2me\ell^\alpha}\otimes Y_j^0, j=1,\ldots,\gamma,\]
where $Y_j^0$ is as in \cite[\S6.A]{Li19}.
Set $R_{m,\alpha,\gamma,\bc,1}=R_{m,\alpha,\gamma,1}\wr A_\bc$, then it is conjugate to the basic subgroup $R_{m,\alpha,\gamma,\bc}$ defined above.
From now on, we will denote any conjugate of the basic subgroup $R_{m,\alpha,\gamma,\bc}$ by the same notation; it should be clear which conjugate is being used by the context.
Note that the definition of $\tau_{m,\alpha,\gamma,\bc}$ and the statement of \cite[Lemma 6.4]{Li19} should be adjusted accordingly.

For any non negative integer $\beta$, denote $\bc_\beta=(1,\dots,1)$ with $\beta$ 1's and set $D_{m,\alpha,\beta}:=R_{m,\alpha,0,\bc_\beta}$.
In \cite{FS89}, $D_{m,\alpha,\beta}$ is denoted by $R_{m,\alpha,\beta}$, but to avoid any confusion of notation, we introduce this new notation.
Then by \cite[(5K)]{FS89}, the defect groups of blocks of $\Sp_{2n}(q)$ are of the form $D_0 \times D_1 \times\cdots\times D_u$ with $D_0$ the trivial group and $D_i=D_{m_i,\alpha_i,\beta_i}$ for $i=1,\cdots,u$.

\subsection{}\label{subsec:partition-symbol}
We now recall some facts about partitions and Lusztig symbols and fix some notation.
In this subsection, $e$ is an arbitrary positive integer (not necessarily a prime).

For a given partition $\lambda$ of some natural number $n$, the $e$-core $\lambda_{(e)}$ and the $e$-quotient $\lambda^{(e)}$ of $\lambda$ is uniquely determined; conversely, the partition is determined by its core and quotient; for details, see \cite[\S3]{Ol93}.
Here, the $e$-core $\lambda_{(e)}$ is again a partition of some natural number $n_0\leq n$, while the $e$-quotient $\lambda^{(e)} = (\lambda_1^{(e)},\ldots,\lambda_e^{(e)})$ is an $e$-tuple (ordered sequence) of partitions.
Given an $e$-core $\kappa$ and an $e$-quotient $Q$, the unique partition $\lambda$ determined by $\lambda_{(e)}=\kappa$ and $\lambda^{(e)}=Q$ is denoted as $\lambda=\kappa*Q$.

For Lusztig symbols, defects and ranks, cores and quotients, degenerate Lusztig symbols, etc., see \cite[\S5]{Ol93}. 
\emph{We remark that degenerate symbols are counted twice when one parametrizes characters, blocks and weights for conformal symplectic groups (see \cite{FS89} and \S\ref{sec:weights} in this paper), while degenerate symbols are counted only once when one parametrizes characters, blocks and weights for symplectic groups (see \cite{Li19}).}
The two copies of the degenerate symbol $\lambda$ are denoted as $\lambda$ and $\lambda'$.
As in \cite{Li19}, the empty Lusztig symbol is viewed as degenerate, thus a non-degenerate symbol is \emph{a fortiori} not empty.
But in many ocasions, the empty Lusztig symbol will be distinguished from the non-empty ones in this paper; the reason for this is that for conformal symplectic groups, the non-empty degenerate symbols are counted twice while the empty symbol is only counted once; see for example Table \ref{tab:block}.

For a given Lusztig symbol $\lambda$ of rank $n$, the $e$-core $\lambda_{(e)}$ and the $e$-quotient $\lambda^{(e)}$ of $\lambda$ is uniquely determined.
But conversely, given an $e$-core $\kappa$ and an $e$-quotient $Q$, there may be one or two Lusztig symbols with $\kappa$ and $Q$ as their $e$-core and $e$-quotient respectively; there are two such symbols if and only if both $\kappa$ and $Q$ are non degenerate.
For details, see \cite[\S5]{Ol93}.
Here, the $e$-core $\lambda_{(e)}$ is again a Lusztig symbol of rank $n_0\leq n$, while the $e$-quotient $\lambda^{(e)}$ is an unordered pair of two $e$-tuples of partitions $[\lambda_1,\dots,\lambda_e;\mu_1,\dots,\mu_e]$, which means that $[\lambda_1,\dots,\lambda_e;\mu_1,\dots,\mu_e]$ and $[\mu_1,\dots,\mu_e;\lambda_1,\dots,\lambda_e]$ are identified; see \cite[\S5]{Ol93}.
An $e$-quotient $[\lambda_1,\dots,\lambda_e;\mu_1,\dots,\mu_e]$ is called degenerate if and only if $(\lambda_1,\dots,\lambda_e)=(\mu_1,\dots,\mu_e)$.

Associated with a given $e$-quotient $Q=[\lambda_1,\dots,\lambda_e;\mu_1,\dots,\mu_e]$, there are one or two ordered sequences of partitions
\[(\lambda_1,\dots,\lambda_e,\mu_1,\dots,\mu_e), \quad (\mu_1,\dots,\mu_e,\lambda_1,\dots,\lambda_e).\]
called the ordered quotient(s) with respect to $Q$ and denoted as $Q_0,Q_0'$.
Thus $Q_0=Q_0'$ if and only if $Q$ is degenerate.

When $\kappa$ is non degenerate, the number of Lusztig symbols with $\kappa$ and $Q$ as their $e$-core and $e$-quotient respectively is exactly the number of ordered quotient(s) with respect to $Q$.
The one or two Lusztig symbols are denoted as $\kappa*Q_0,\kappa*Q_0'$.
Equivalently, the (necessarily non degenerate) Lusztig symbols with non degenerate cores are determined uniquely by the cores and the ordered quotients.

When $\kappa$ is degenerate, then for any given quotient $Q$, there is only one Lusztig symbol $\lambda$ with $\kappa$ and $Q$ as their $e$-core and $e$-quotient respectively.
When $Q$ is non degenerate, $\lambda$ is non degenerate, and we set $\lambda=\kappa*Q_0=\kappa*Q_0'$.
When $Q$ is degenerate (thus we can identify $Q$ with its associated ordered quotient), $\lambda$ is degenerate.
Note that $\lambda$ should be counted twice when one considers conformal groups.
In this case, the two copies of the degenerate symbols are denoted as $\lambda=\kappa*(Q,0)$ and $\lambda'=\kappa*(Q,1)$.

Note that our convention of notation used to consider characters and weights of conformal symplectic groups is slightly different from those in \cite{Li19} used to consider characters and weights of symplectic groups.

\subsection{}\label{subsec:action-tz}
Assume $q$ is odd.
Now, we recall the Jordan decomposition of characters of $\tG$ and consider the action of $\Irr(\tG/G)$ on $\Irr(\tG)$.

\begin{thm}[Lusztig, {\cite[Theorem 15.8]{CE04}}]\label{thm:Jordan}
Assume $\tbG$, $\tG$, $\tbG^*$, $\tG^*$ are as in \S\ref{sect:gLt}.
There is a bijection for any $\ts\in\tG^*$ between Lusztig series (note that $\tbG$ has connected center):
\[ \cJ_{\ts}: \quad \cE(\tG,\ts) \longleftrightarrow \cE(C_{\tG^*}(\ts),1) \]
such that
\[ \Group{\tchi,R_{\tbT}^{\tbG}\hat{\ts}}_{\tG} = \varepsilon_{\tbG}\varepsilon_{C_{\tbG^*}(\ts)} \Group{\cJ_{\ts}(\tchi),R_{\tbT^*}^{C_{\tbG^*}(\ts)}1}_{C_{\tG^*}(\ts)}, \]
where the $\tG$-conjugacy class of $(\tbT,\hat{\ts})$ corresponds to the $\tG^*$-conjugacy class of $(\tbT^*,\ts)$ via duality; see \cite[Theorem 8.21]{CE04}.
Furthermore $\tchi\in\cE(\tG,\ts)$ is uniquely determined by the scalar products $\Group{\tchi,R_{\tbT}^{\tbG}\hat{\ts}}_{\tG}$.
The set $\Irr(\tG)$ of characters of $\tG$ has a decomposition
\[ \Irr(\tG) = \bigcup\limits_{\ts} \cE(\tG,\ts), \]
where $\ts$ runs over a $\tG^*$-transversal of semisimple elements of $\tG^*$.
\end{thm}

Let $s=i^*(\ts)$, then $C_{\bG^*}^\circ(s)^F$ is as in \cite[\S3.A]{Li19}.
We can identify $\cE(C_{\tG^*}(\ts),1)$ with $\cE(C_{\bG^*}^\circ(s)^F,1)$, which can be parametrized by $\lambda=\prod_\Gamma\lambda_\Gamma$, where $\lambda_\Gamma$ is a partition of $m_\Gamma(s)$ for $\Gamma\in\cF_1\cup\cF_2$
while $\lambda_\Gamma$ is a Lusztig symbol of rank $[\frac{m_\Gamma(s)}{2}]$ for $\Gamma\in\cF_0$; see \cite[\S13.8]{Car85}.
Recall that the degenerate symbols in component $X+1$ are counted twice.
For $\lambda$ as above, $\lambda'$ is defined as below: $\lambda'_\Gamma=\lambda_\Gamma$ when $\lambda_\Gamma$ is a partition or non degenerate symbol and $(\lambda')_\Gamma = (\lambda_\Gamma)'$ when $\lambda_\Gamma$ is a degenerate symbol.
Denote the characters in $\cE(C_{\tG^*}(\ts),1)$ and $\cE(\tG,\ts)$ corresponding to $\lambda$ by $\chi_\lambda$ and $\chi_{\ts,\lambda}$ respectively.

\begin{rem}\label{rem:action-tz}
The duality induces an isomorphism of abelian groups (see \cite[(8.19)]{CE04}):
\[ Z(\tG^*) \to \Irr(\tG/G), \quad \tz \mapsto \hat{\tz}. \]
We consider the action of $\Irr(\tG/G)$ on $\Irr(\tG)$.
Note that by \cite[(8.20)]{CE04}, $\hat{\tz}\chi_{\ts,\lambda} \in \cE(\tG,\tz\ts)$.
Then
\begin{align*}
\Group{\hat{\tz}\chi_{\ts,\lambda},R_{\tbT}^{\tbG}\widehat{\tz\ts}}_{\tG}
&= \Group{\chi_{\ts,\lambda},\hat{\tz}^{-1}R_{\tbT}^{\tbG}\widehat{\tz\ts}}_{\tG}
= \Group{\chi_{\ts,\lambda},R_{\tbT}^{\tbG}\hat{\tz}^{-1}\widehat{\tz\ts}}_{\tG}
= \Group{\chi_{\ts,\lambda},R_{\tbT}^{\tbG}\hat{\ts}}_{\tG} \\
&= \varepsilon_{\tbG}\varepsilon_{C_{\tbG^*}(\ts)} \Group{\chi_\lambda,R_{\tbT^*}^{C_{\tbG^*}(\ts)}1}_{C_{\tG^*}(\ts)}.
\end{align*}
On the other hand, note that $C_{\tbG^*}(\tz\ts)=C_{\tbG^*}(\ts)$, then
\[ \Group{\chi_{\tz\ts,\lambda},R_{\tbT}^{\tbG}\widehat{\tz\ts}}_{\tG}
= \varepsilon_{\tbG}\varepsilon_{C_{\tbG^*}(\ts)} \Group{\chi_\lambda,R_{\tbT^*}^{C_{\tbG^*}(\ts)}1}_{C_{\tG^*}(\ts)}. \]
Thus with the Jordan decomposition of $\Irr(\tG)$ as in Theorem \ref{thm:Jordan}, $\hat{\tz}\chi_{\ts,\lambda} = \chi_{\tz\ts,\lambda}$.

\paragraph{}
Fix a generator $\tz_0$ of $Z(\tG^*)$, then $\tz_0^{(q-1)/2}=-e$, where $e$ is the identity element of the special Clifford group $\D_0(V^*)$.
For each $G^*$-conjugacy class of semisimple elements, fix a representative $s$ and a preimage $\ts$ of $s$ under $i^*$.

\begin{enumerate}[(1)]\setlength{\itemsep}{0pt}

\item If $m_{X+1}(s)=0$, then ${i^*}^{-1}(\Cl_{G^*}(s))$ consists of $q-1$ conjugacy classes of $\tG^*$, each of which contains exactly one element $\tz\ts$ of $\F_q^\times\ts$ by \cite[(2D)]{FS89}.
Then $Z(\tG^*) \cong \Irr(\tG/G)$ acts regularly on $\set{ \chi_{\tz\ts,\lambda} \mid \tz \in Z(\tG^*)}$ as follows
\[
\chi_{\ts,\lambda} \xmapsto{\widehat{\tz_0}\cdot}  \chi_{\tz_0\ts,\lambda}  \xmapsto{\widehat{\tz_0}\cdot}  \cdots   \xmapsto{\widehat{\tz_0}\cdot}  \chi_{\tz_0^{q-2}\ts,\lambda} \xmapsto{\widehat{\tz_0}\cdot} \chi_{\ts,\lambda}.
\]

\item If $m_{X+1}(s)\neq0$, then ${i^*}^{-1}(\Cl_{G^*}(s))$ consists of $(q-1)/2$ conjugacy classes of $\tG^*$, each of which contains exactly two elements $\tz\ts,-\tz\ts$ of $\F_q^\times\ts$ by \cite[(2D)]{FS89}.

\begin{enumerate}\setlength{\itemsep}{-2pt}
\item If furthermore $\lambda_{X+1}$ is non degenerate, then by \cite[(4C)]{FS89}, $\widehat{-e}\chi_{\ts,\lambda}=\chi_{\ts,\lambda}$.
So, with the Jordan decomposition of $\Irr(\tG)$ as in Theorem \ref{thm:Jordan}, we have $\chi_{-\ts,\lambda}=\chi_{\ts,\lambda}$.
Thus $Z(\tG^*) \cong \Irr(\tG/G)$ acts transitively on $\set{ \chi_{\tz\ts,\lambda} \mid \tz \in Z(\tG^*)}$ with kernel $\group{-e}$ as follows
\[
\chi_{\ts,\lambda} \xmapsto{\widehat{\tz_0}\cdot}  \chi_{\tz_0\ts,\lambda}  \xmapsto{\widehat{\tz_0}\cdot}  \cdots   \xmapsto{\widehat{\tz_0}\cdot}  \chi_{\tz_0^{(q-1)/2-1}\ts,\lambda} \xmapsto{\widehat{\tz_0}\cdot} \chi_{\ts,\lambda}.
\]

\item If furthermore $\lambda_{X+1}$ is degenerate, then $\widehat{-e}\chi_{\ts,\lambda}=\chi_{\ts,\lambda'}$ by \cite[(4C)]{FS89}.
Thus with the Jordan decomposition of $\Irr(\tG)$ as in Theorem \ref{thm:Jordan}, we have $\chi_{-\ts,\lambda}=\chi_{\ts,\lambda'}$.
Then $Z(\tG^*) \cong \Irr(\tG/G)$ acts regularly on $\set{ \chi_{\tz\ts,\lambda},\chi_{\tz\ts,\lambda'} \mid \tz \in Z(\tG^*)}$ and we can choose the labels such that
\[
\chi_{\ts,\lambda} \xmapsto{\widehat{\tz_0}\cdot}  \cdots   \xmapsto{\widehat{\tz_0}\cdot}  \chi_{\tz_0^{(q-1)/2-1}\ts,\lambda} \xmapsto{\widehat{\tz_0}\cdot} \chi_{\ts,\lambda'} \xmapsto{\widehat{\tz_0}\cdot}  \cdots   \xmapsto{\widehat{\tz_0}\cdot}  \chi_{\tz_0^{(q-1)/2-1}\ts,\lambda'} \xmapsto{\widehat{\tz_0}\cdot} \chi_{\ts,\lambda}.
\]

\end{enumerate}
\end{enumerate}


\end{rem}


\section{Brauer pairs and blocks of conformal groups}\label{sec:prepare}

From now on until the last section, we assume $q$ is odd.
Before considering the weights of conformal symplectic groups, we make some preparations, including classification of radical subgroups, some considerations of Brauer pairs and blocks.


\subsection{}\label{subsec:Rad-tG} 
We begin by dealing with the conjugacy classes of radical subgroups of $\tG$.
Denote by $Z(\tG)_\ell$ the $\ell$-part of the cyclic group $Z(\tG)$.
For any symplectic space $V$, denote by $\I_0(V)$ and $\J_0(V)$ the symplectic group and the conformal symplectic group on $V$ respectively.
Recall that for any radical subgroup $R= R_0 \times R_1 \times\cdots\times R_u$ with $R_i=R_{m_i,\alpha_i,\gamma_i,\bc_i}$ for $i=1,\ldots,u$, there is a corresponding decomposition of the space $V= V_0\perp V_1\perp\cdots\perp V_u$.

\begin{prop}\label{prop:Rad(tG)}
For any radical subgroup $R$ of $G$ as above, denote $\tR:=Z(\tG)_\ell R$.
\begin{enumerate}[(1)]\setlength{\itemsep}{0pt}
\item $\tN:=N_{\tG}(\tR)=N_{\tG}(R)=\group{\tau,N}$ with $N:=N_G(R)$ and $\tau=\tau_0\times\tau_1\times\cdots\times\tau_u$, where $\tau_0$ is an element of order $q-1$ generating $\J_0(V_0)$ modulo $\I_0(V_0)$ and $\tau_i=\tau_{m_i,\alpha_i,\gamma_i,\bc_i}$ for $i=1,\ldots,u$.
\item The map $R \mapsto \tR$ defines a bijection between $\Rad(G)$ and $\Rad(\tG)$ with inverse $\tR \mapsto G\cap\tR$, which induces a bijection between $\Rad(G)/\simi{G}$ and $\Rad(\tG)/\simi{\tG}$.
\end{enumerate}
\end{prop}
\begin{proof}
The first assertion follows from \cite[Lemma 6.6 (4)]{Li19}.
Let $\hG=Z(\tG)G$.
Since $|\tG:\hG|=2$ and $\ell$ is odd, $\Rad(\tG)=\Rad(\hG)$.
Now $\hG$ is the central product of $Z(\tG)$ and $G$ over $Z(G)$ and $|Z(G)|=2$, the maps given in (2) are bijections between $\Rad(G)$ and $\Rad(\hG)$.
By the definition, $\tau$ generates $\tG$ modulo $G$.
Thus $|\tG:N_{\tG}(R)|=|G:N_G(R)|$ and the given maps induce also a bijection between conjugacy classes.
\end{proof}

Many results about the defect groups of blocks of $\tG$ in \cite{FS89} hold for radical subgroups of $\tG$.
In the sequel, when we state such a result for all radical subgroups, we will often refer to \cite{FS89} for the corresponding result for defect groups.
We will also give new proofs for some of these statements using explicit calculations in twisted groups.


\subsection{Brauer pairs}\label{subsec:Brauerpair}
Let $R$ be a radical subgroup of $G$ as above.
Set $R_+=R_1 \times\cdots\times R_u$, then $R=R_0\times R_+$, $V = V_0 \perp V_+$ and all related constructions can be decomposed correspondingly.
Then $C:=C_G(R)= C_0 \times C_+$, where $C_0=\I_0(V_0)$ and $C_+=C_1 \times\cdots\times C_u$ with  $C_i=C_{m_i,\alpha_i,\gamma_i,\bc_i}$ for $i>0$; for $C_{m_i,\alpha_i,\gamma_i,\bc_i}$, see the definition before \cite[Lemma 6.4]{Li19}.
Let $\tau$ be as before, then by \cite[Lemma 6.4]{Li19}, $\tC:=C_{\tG}(R)=C_{\tG}(\tR)=\group{\tau,C}$, $[\tau,RC]=1$ and $\tau^{q-1}\in Z(C)$.
The following is a generalization of \cite[(5J)]{FS89} to radical subgroups.

\begin{lem}\label{lem:cen-prod}
Let $R$ be a radical subgroup of $G$, $\tR=Z(\tG)_\ell R$ and $Z_+$ be a subgroup of $Z(C_+)$ containing $\tau_+^{q-1}$.
\begin{enumerate}[(1)]\setlength{\itemsep}{0pt}
\item $\tC$ is a central product of $\group{\tau,C_0Z_+}$ and $C_+$ over $Z_+$.
\item Assume $Z_+=Z(C_+)$, then $Z(\tG)_\ell\leq\group{\tau,C_0Z(C_+)}$.
	If $g\in N$ and $g_0,g_+$ are the restriction of $g$ to $V_0,V_+$ respectively, then $[\tau_0,g_0]\in C_0, [\tau_+,g_+]\in Z(C_+)$.
	In particular, $N$ and thus $\tN$ normalize $\group{\tau,C_0Z(C_+)}$ and $C_+$.
\end{enumerate}
\end{lem}
\begin{proof}
This can be proved by the same method as in \cite[(5J)]{FS89}, or can be proved with explicit calculation by transferring to the twisted group $\tG^{tw}$ and using the choice of $\tau_{m,\alpha,\gamma,\bc}^{tw}$ in \cite[(6.3)]{Li19}.
\end{proof}

\begin{rem}\label{rem:cen-prod}
As in \cite{FS89}, mainly the following two cases of $Z_+$ are used:
\begin{enumerate}[(1)]\setlength{\itemsep}{0pt}
\item $Z_+=\group{\tau_+^{q-1}}$, in which case,
	$\tC$ is viewed as a central product of $\group{\tau,C_0}$ and $C_+$ over $\group{\tau_+^{q-1}}$.
\item $Z_+=Z(C_+)$, in which case, we denote $\tC_0=\group{\tau,C_0Z(C_+)}$.
	As shown in the above lemma, the advantage of this case is that $\tN$ stabilize this central product decomposition.
\end{enumerate}
\end{rem}

Let $\cF'$ be the subset of polynomials in $\cF$ whose roots have $\ell'$-order.
For any $\Gamma\in\cF'$, $G_\Gamma,R_\Gamma,C_\Gamma,N_\Gamma,\theta_\Gamma$ are as on \cite[p.22]{An94}.
We refer to \cite[\S6.B]{Li19} for the constructions for symplectic groups.
In particular, $\Gamma$ determines an $N_\Gamma$-conjugacy class in $\dz(C_\Gamma/R_\Gamma)$ and $\theta_\Gamma$ is a chosen one in this conjugacy class.
Defined as in \cite[\S6.B]{Li19} are also $R_{\Gamma,\gamma,\bc}$, $\theta_{\Gamma,\gamma,\bc}$, the set $\cR_{\Gamma,\delta}=\set{R_{\Gamma,\delta,1},R_{\Gamma,\delta,2},\ldots}$, $\theta_{\Gamma,\delta,i}$, and the Notation 6.10.
Recall that $D_{m,\alpha,\beta}$ is defined at the end of \S\ref{subsec:ano-conju}.
Then similarly, we denote $D_{\Gamma,\beta}=D_{m_\Gamma,\alpha_\Gamma,\beta}$.
So $D_{\Gamma,\beta}=R_\Gamma\wr A_{\bc_\beta}$.

Let $\theta\in\dz(C/Z(R))$, then $\theta=\theta_0\times\theta_+$, where $\theta_0$ is a character of $\I_0(V_0)$ of defect zero, while $\theta_+$ and $R_+$ can be decomposed as follows:
\begin{equation}\label{equ:theta_+R_+}
\theta_+ = \prod_{\Gamma,\delta,i} \theta_{\Gamma,\delta,i}^{t_{\Gamma,\delta,i}},\quad
R_+=\prod_{\Gamma,\delta,i}R_{\Gamma,\delta,i}^{t_{\Gamma,\delta,i}}.
\addtocounter{thm}{1}\tag{\thethm}
\end{equation}

Recall that $\tR=Z(\tG)_\ell R$.
Then by Lemma \ref{lem:cen-prod}, the results for canonical characters associated with defect groups in \cite[\S7]{FS89} can be generalized to radical subgroups.
Specifically, any $\ttheta\in\dz(\tC/Z(\tR)\mid\theta_R)$ is of the form
\[\ttheta=\ttheta_0\theta_+,\]
where $\ttheta_0$ is the canonical character of a block of $\tC_0=\group{\tau,C_0Z(C_+)}$ of the central defect group $Z(\tR)=Z(\tG)_\ell Z(R)$, and the linear characters of $Z(C_+)$ induced by $\ttheta_0$ and $\theta_+$ are the same.
Note that $\ttheta_0\in\Irr(\tC_0\mid\theta_0)$.
Recall that
\begin{equation}\label{equ:cen-prod-tC-R}
\tC_0 = \group{\tau,C_0} \times_{\group{\tau_+^{q-1}}} Z(C_+)
\addtocounter{thm}{1}\tag{\thethm}
\end{equation}
is a central product over $\group{\tau_+^{q-1}}$.

Let $D'=D'_0\times D'_+$, $\theta'=\theta'_0\times\theta'_+$ with $D'_0=R_0$, $\theta'_0=\theta_0$ and
\[D'_+= \prod_\Gamma R_\Gamma^{m_\Gamma},\quad \theta'_+=\prod_\Gamma\theta_\Gamma^{m_\Gamma},\]
where, $m_\Gamma = \sum_{\delta,i}t_{\Gamma,\delta,i}\ell^\delta$.
Let $\tD'=Z(\tG)_\ell D'$, $C',\tC'$ denote the centralizers of $D',\tD'$ in $G,\tG$ respectively and $C' = C'_0 \times C'_+$ with $C'_0=C_0=\I_0(V_0)$.
Then $\theta' \in \dz(C'/Z(D'))$.
Assume $\ttheta'=\ttheta'_0\theta'_+\in\dz(\tC'/Z(\tD')\mid\theta')$ and $\theta'$ are in the same relation as $\ttheta$ and $\theta$.
Note that
\begin{equation}\label{equ:cen-prod-tC-D'}
\tC'_0:= \group{\tau,C'_0Z(C'_+)} = \group{\tau,C'_0} \times_{\group{\tau_+^{q-1}}} Z(C'_+)
\addtocounter{thm}{1}\tag{\thethm}
\end{equation}
is a central product over $\group{\tau_+^{q-1}}$.
Since $Z(C_+) \leq Z(C'_+)$, $\ttheta'_0$ can be chosen to be an extension of $\ttheta_0$ by (\ref{equ:cen-prod-tC-R}) and (\ref{equ:cen-prod-tC-D'}).

Let $D=D_0\times D_+$, $\theta_D=\theta_{D,0}\times\theta_{D,+}$ with $D_0=R_0$, $\theta_{D,0}=\theta_0$ and
\[D_+= \prod_{\Gamma,\beta} D_{\Gamma,\beta}^{r_{\Gamma,\beta}},\quad \theta_{D,+}=\prod_{\Gamma,\beta}(\theta_\Gamma\otimes I_{\ell^\beta})^{r_{\Gamma,\beta}},\]
where, $m_\Gamma = \sum_\beta r_{\Gamma,\beta}\ell^\beta$ is the $\ell$-adic decomposition of $m_\Gamma$.
Let $\tD=Z(\tG)_\ell D$, $C_D,\tC_D$ denote the centralizers of $D,\tD$ in $G,\tG$ respectively and $C_D = C_{D,0} \times C_{D,+}$ with $C_{D,0}=C_0=\I_0(V_0)$.
Then $\theta_D \in \dz(C_D/Z(D))$.
Assume $\ttheta_D=\ttheta_{D,0}\theta_{D,+}\in\dz(\tC_D/Z(\tD)\mid\theta_D)$ and $\theta_D$ are in the same relation as $\ttheta$ and $\theta$.
Similarly as above, 
\begin{equation}\label{equ:cen-prod-tC-D}
\tC_{D,0} = \group{\tau,C_{D,0}Z(C_{D,+})} = \group{\tau,C_{D,0}} \times_{\group{\tau_+^{q-1}}} Z(C_{D,+})
\addtocounter{thm}{1}\tag{\thethm}
\end{equation}
is a central product over $\group{\tau_+^{q-1}}$.
Since $Z(C_{D,+}) \leq Z(C'_+)$, $\ttheta_{D,0}$ can be chosen to be the restriction of $\ttheta'_0$.

\begin{rem}
When every basic subgroup in $R$ is chosen to be the conjugate in \S\ref{subsec:ano-conju}, the element $\tau=\iota(\tau^{tw})$ for $R,D',D$ can be chosen to be the same one and we make this choice in the sequel.
\end{rem}

\begin{prop}\label{prop:Brauerpair}
With the above notation,
\begin{enumerate}[(1)]\setlength{\itemsep}{0pt}
\item $(D',\theta') \unlhd (D,\theta_D)$ as Brauer pairs in $G$.
	$(D,\theta_D)$ is a maximal pair for some block $B$ of $G$ and all maximal pairs are of this form.
\item $(\tD',\ttheta') \unlhd (\tD,\ttheta_D)$ as Brauer pairs in $\tG$.
	$(\tD,\ttheta_D)$ is a maximal pair for some block $\tB$ of $\tG$ and all maximal pairs are of this form.
\item $\tB$ covers $B$.
\item $(R,\theta) \leq (D,\theta_D)$ as Brauer pairs in $G$ and $(\tR,\ttheta) \leq (\tD,\ttheta_D)$ as Brauer pairs in $\tG$.
\end{enumerate}
\end{prop}

\begin{proof}
(1) and (2) are just \cite[(8A)]{FS89}.
For (3), note first that $N_{\tG}(\tD)=N_{\tG}(D)$.
Let $b,\tb$ be the Brauer correspondents of $B,\tB$ respectively, then by the definition, $\tb$ covers $b$.
Thus (3) follows from the Harris-Kn\"orr correspondence \cite{HK85}.
To prove $(R,\theta) \leq (D,\theta_D)$ as Brauer pairs in $G$, it suffices to prove that $(R_+,\theta_+) \leq (D_+,\theta_{D,+})$ as Brauer pairs in $G_+=\I_0(V_+)$.
Let $z_+$ be an element in $Z(D_+)$ such that (i) $o(z_+)=\ell$, (ii) $[z_+,V_+]=V_+$, (iii) $z_+$ is primary; see \cite[p.178]{FS89}.
Then by the above constructions from $R$ to $D$, $z_+\in Z(R_+)$.
So both $R_+C_{G_+}(R_+)$ and $D_+C_{G_+}(D_+)$ are contained in $C_{G_+}(z_+)$, which is isomorphic to a general linear or unitary group when $\ell$ is a linear or unitary prime respectively.
By the construction of $D$ and results of \cite{Brou86}, $(R_+,\theta_+) \leq (D_+,\theta_{D,+})$ as Brauer pairs in $C_{G_+}(z_+)$ and so as Brauer pairs in $G_+=\I_0(V_+)$; see also \cite[p.18]{An94}.
Then $(\tR,\ttheta) \leq (\tD,\ttheta_D)$ as Brauer pairs in $\tG$ follows by a similar argument as in \cite[(8A)]{FS89} using the first kind of central product decompositions in Remark \ref{rem:cen-prod}.
\end{proof}

\begin{cor}\label{cor:weights-into-blocks}
Keep the above notation and assume $(\tR,\tvarphi)$ is a weight such that $\tvarphi \in \Irr(\tN\mid\ttheta)$, then $(\tR,\tvarphi)$ belongs to the block $\tB$ in the above proposition.
\end{cor}

\subsection{}\label{subsec:block}

Let $R$ be a radical subgroup of $G$ and keep the above constructions and notation.
In \cite[\S11]{FS89}, the authors use the Brauer pair $(\tD',\ttheta')$ to give a label $(\ts,\cK)$ for the block $\tB$.
Recall that we use $\ts$ to denote a semisimple element in $\tG^*$ and denote by $s=i^*(\ts)$ the image of $\ts$ in $G^*$, while in \cite{FS89}, the corresponding notation is $s$ and $\bar{s}$; some other notation used here is also slightly different from the one in \cite{FS89}.
For convenience, we first recall the results as follows.

Since $\tC'$ is a Levi subgroup of $\tG$, $\ttheta'$ is contained in some Lusztig series $\cE(\tC',(\ts))$ with $\ts$ a semisimple $\ell'$-element in $(\tC')^* \leq \tG^*$.
This $\ts$ is the semisimple part of the label $(\ts,\cK)$.

By \cite[(11.2),(11.3)]{FS89},
\begin{equation}\label{equ:tC'_0}
\begin{aligned}
\tC'_0/\I_0(V_0) &\cong \group{\tau_+,Z(C'_+)},\\
\tC'_0/Z(C'_+) &\cong \group{\tau_0,\I_0(V_0)} = \J_0(V_0).
\end{aligned}
\addtocounter{thm}{1}\tag{\thethm}
\end{equation}
Then $\ttheta'$ can be decomposed as $\ttheta'=\tchi'\zeta'\theta'_+$, where $\tchi'$ is a character of $\J_0(V_0)$ and $\zeta'$ is an extension to $\group{\tau_+,Z(C'_+)}$ of the linear character $\zeta'_+$ of $Z(C'_+)$ induced by $\theta'_+$ and $\ttheta'_0$.
By the Jordan decomposition of characters of $\J_0(V_0)$, $\tchi'=\tchi_{\ts_0,\kappa}^{\J_0(V_0)}$ for some semisimple $\ell'$-element $\ts_0$ of $\J_0(V_0)^*$ and $\kappa=\prod_{\Gamma\in\cF'}\kappa_\Gamma$, where $\kappa_\Gamma$ is a Lusztig symbol or partition according to $\Gamma\in\cF_0$ or $\Gamma\notin\cF_0$.
But since the decomposition $\ttheta'=\tchi'\zeta'\theta'_+$ is not unique, $\kappa$ has different choices.
By the process described in \cite[\S11]{FS89}, a carefully chosen set $\cK$ of the unipotent labels $\kappa$ appearing in the label of $\tchi'$ in some decompositions $\ttheta'=\tchi'\zeta'\theta'_+$ is the unipotent label of the block $\tB$.
For symplectic groups, $|\cK|$ is one or two.

In \cite[\S11, \S13]{FS89}, the authors consider both conformal symplectic and conformal orthogonal groups.
Many of their complicated considerations are for conformal orthogonal groups.
For convenience, we list in Table \ref{tab:block} their results about conformal symplectic groups and the results in \cite{Li19} for symplectic groups.

\begin{table}\linespread{1.5}
\centering
\setlength{\tabcolsep}{5pt}
\footnotesize
\caption{Blocks of $G$ and $\tG$}

\begin{tabular}{cc|ccccc}
\\
\toprule

cases & conditions & $\cB$ & $\tcB$ & $\Irr(\tB)\cap\cE(\tG,\ell')$ \\

\midrule
       
(I) & $m_{X+1}(s)=0$
& $B_{s,\kappa}$
& $\begin{array}{c} \tB_{\ts\tz,\set{\kappa}},\\ \tz \in Z(\tG^*)_{\ell'}\\ \end{array}$
& $\begin{array}{c} \chi_{\ts\tz,\lambda},\\ \set{\kappa}=\textrm{ core of }\lambda\\ \end{array}$\\

\hline

(II) & $\begin{array}{c} \kappa_{X+1}\neq\O\\ \textrm{non-degenerate}\\ 2\rk\kappa_{X+1}=m_{X+1}(s)\\ \end{array}$
& $\begin{array}{c} B_{s,\kappa,i}, \\ i=0,1 \\ \end{array}$
& $\begin{array}{c} \tB_{\ts\tz,\set{\kappa}},\\ \tz \in Z(\tG^*)_{\ell'}\\ \end{array}$
& $\begin{array}{c} \chi_{\ts\tz,\lambda},\\ \set{\kappa}=\textrm{ core of } \lambda\\ \end{array}$ \\

\hline

(III) & $\begin{array}{c} \kappa_{X+1}\neq\O\\ \textrm{degenerate}\\ 2\rk\kappa_{X+1}=m_{X+1}(s)\\ \end{array}$
& $B_{s,\kappa}$
& $\begin{array}{c} \tB_{\ts\tz,\cK},\\ \tz \in Z(\tG^*)_{\ell'}\\ \cK=\set{\kappa}\textrm{~or~}\set{\kappa'}\\ \end{array}$
& $\begin{array}{c} \chi_{\ts\tz,\lambda},\\ \cK=\textrm{ core of } \lambda\\ \end{array}$ \\

\hline

(IV) & $\begin{array}{c} \kappa_{X+1}=\O\\ m_{X+1}(s)\neq0\\ \end{array}$
& $B_{s,\kappa}$
& $\begin{array}{c} \tB_{\ts\tz,\set{\kappa}},\\ \tz \in Z(\tG^*)_{\ell'}\\ \end{array}$
& $\begin{array}{c} \chi_{\ts\tz,\lambda},\\ \set{\kappa}=\textrm{ core of } \lambda\\ \end{array}$ \\

\hline

(V) & $\begin{array}{c} \kappa_{X+1}\neq\O\\ \textrm{non-degenerate}\\ 2\rk\kappa_{X+1}<m_{X+1}(s) \end{array}$
& $\begin{array}{c} B_{s,\kappa,i}, \\ i=0,1 \\ \end{array}$
& $\begin{array}{c} \tB_{\ts\tz,\set{\kappa}},\\ \tz \in Z(\tG^*)_{\ell'}\\ \end{array}$
& $\begin{array}{c} \chi_{\ts\tz,\lambda},\\ \set{\kappa}=\textrm{ core of } \lambda\\ \end{array}$ \\

\hline

(VI) & $\begin{array}{c} \kappa_{X+1}\neq\O\\ \textrm{degenerate} \\ 2\rk\kappa_{X+1}<m_{X+1}(s)\\ \end{array}$
& $B_{s,\kappa}$
& $\begin{array}{c} \tB_{\ts\tz,\set{\kappa,\kappa'}},\\ \tz \in Z(\tG^*)_{\ell'}\\ \end{array}$
& $\begin{array}{c} \chi_{\ts\tz,\lambda},\\ \set{\kappa,\kappa'}=\textrm{ core of } \lambda\\ \end{array}$ \\

\bottomrule
\end{tabular}\label{tab:block}
\end{table}

\begin{rem}\label{rem:block}
We give some remarks about the notation in Table \ref{tab:block}.
\begin{enumerate}[(1)]\setlength{\itemsep}{-2pt}
\item $\cB$ is a $\tG$-conjugacy class of blocks of $G$, i.e. covered by a certain block of $\tG$; $\tcB=\Bl(\tG\mid\cB)$.
	These sets satisfy the requirements in Theorem \ref{thm:criterion}.
\item Recall that for symplectic groups, degenerate Lusztig symbols are counted once.
	See \cite[Theorem 3.2, Theorem 3.7]{Li19} for notation of characters and blocks for symplectic groups.
\item For conformal symplectic groups, every degenerate Lusztig symbol is counted twice and $\kappa'$ is defined as in \S\ref{subsec:action-tz};
	see \cite[p.132]{FS89}.
	For the definition of the core of $\lambda$ in the case of conformal groups, see \cite[\S9]{FS89}.
	The definition of $e_\Gamma$ in \cite[\S9]{FS89} is the same as that before \cite[Theorem 3.7]{Li19}.
\end{enumerate}
\end{rem}

Set 
\[i\IBr(\tB_{\ts,\cK}) = \Set{ \lambda=\prod_\Gamma\lambda_\Gamma ~\middle|~ \cK \textrm{~is the core of~} \lambda }.\]
Then by \cite{Ge93} and \cite[\S13]{FS89}, the following holds.
\begin{lem}\label{lem:IBr}
There are bijections $i\IBr(\tB_{\ts,\cK}) \leftrightarrow \Irr(\tB_{\ts,\cK}) \cap \cE(\tG,\ts) \leftrightarrow \IBr(\tB_{\ts,\cK})$, denoted by $\lambda \mapsto \chi_{\ts,\lambda} \mapsto \phi_{\ts,\lambda}$.
\end{lem}


\section{The blockwise Alperin weight conjecture for $\tG$}\label{sec:weights}

In this section, we classify the weights of $\tG$ and prove the blockswise Alperin weight conjecture for $\tG$.

\subsection{}\label{subsec:N(ttheta)}
Let $(\tR,\tvarphi)$ be a weight of $\tG$ belonging to the block $\tB$ with label $(\ts,\cK)$ as above.
Then $\tvarphi$ is of the form $\tvarphi=\Ind_{\tN(\ttheta)}^{\tN}\tpsi$, where $\ttheta\in\dz(\tC/Z(\tR))$ with $(\tR,\ttheta)$ a $\tB$-Brauer pair and $\tpsi\in\dz(\tN(\ttheta)/\tR\mid\ttheta)$.
When $\ttheta$ runs over a set of representatives of $\tN$-conjugacy classes of $\dz(\tC/Z(\tR))$, the above construction gives all weight characters associated with $\tR$.
Thus we need to consider the $\tN$-conjugacy classes of $\ttheta$ and $\tN(\ttheta)$.

We first show how a result concerning $R_{\Gamma,\gamma}$ can be derived from \cite[(6A)]{FS89} which deals only with $R_\Gamma$.
But we will state and prove it in the twisted group $G_{\Gamma,\gamma}^{tw}$.
Here, recall that for all related constructions, we replace the two parameters $m_\Gamma,\alpha_\Gamma$ by $\Gamma$.
Let $s_\Gamma$ and $s_\Gamma^*$ be as in \cite[\S6.B]{Li19}.
Let $s_\Gamma^0, \theta_\Gamma^0 = \pm R_{T_\Gamma^0}^{C_\Gamma^0} \hat{s_\Gamma^0}, \theta_{\Gamma,\gamma}^0 = \theta_\Gamma^0 \otimes I_{\ell^\gamma}, \theta_{\Gamma,\gamma}^{tw}$ be as in \cite[\S6.C]{Li19}.
Set $\phi_\Gamma^0 = \hat{s_\Gamma^0}$ and denote by $\phi_\Gamma^{tw}$ the induced character on $Z(C_{\Gamma,\gamma}^{tw}) = \hbar(Z(C_\Gamma^0) \otimes I_{\ell^\gamma}) = \hbar(\fZ_{q^{e\ell^{\alpha_\Gamma}}-\varepsilon}I_{m_\Gamma\ell^\gamma})$, where $\fZ_{q^{e\ell^{\alpha_\Gamma}}-\varepsilon}$ is the cyclic subgroup of $\barF_p^\times$ of order $q^{e\ell^{\alpha_\Gamma}}-\varepsilon$.

\begin{lem}\label{lem:phi-Gamma}
Let $g\in N_{\Gamma,\gamma}^{tw}(\theta_{\Gamma,\gamma}^{tw})$ and $z=[\tau_{\Gamma,\gamma}^{tw},g]$, then $z\in Z(C_{\Gamma,\gamma}^{tw})$ and the following hold:
\begin{enumerate}[(1)]\setlength{\itemsep}{0pt}
\item $\phi_\Gamma^{tw}(z)=\pm1$.
\item If $\Gamma\neq X+1$, then $\phi_\Gamma^{tw}(z)=1$.
\item Assume $\Gamma=X+1$, then $\phi_{X+1}^{tw}(z)=1$ or $-1$
	according to whether $g$ is a square or a non-square in $N_{\Gamma,\gamma}^{tw}/\hbar(N_{\Gamma,\gamma}^0)$.
\end{enumerate}
\end{lem}

\begin{proof}
Recall from \cite[Lemma 6.6]{Li19} that $N_{\Gamma,\gamma}^{tw} = \hbar(N_{\Gamma,\gamma}^0) V_{\Gamma,\gamma}^{tw}$, where $N_{\Gamma,\gamma}^0=C_{\Gamma,\gamma}^0L_{\Gamma,\gamma}^0$ is a central product and $V_{\Gamma,\gamma}^{tw}$ is as in \cite[Lemma 6.6 (2)]{Li19}.
Also, $\tau_{\Gamma,\gamma}^{tw}$ is as \cite[(6.3)]{Li19}, then by the definition of $\hbar$ in \cite[\S6.A]{Li19}, $[\tau_{\Gamma,\gamma}^{tw},N_{\Gamma,\gamma}^{tw}] = [\tau_{\Gamma,\gamma}^{tw},V_{\Gamma,\gamma}^{tw}] \leq Z(C_{\Gamma,\gamma}^{tw})$.
By the structure of $\tau_{\Gamma,\gamma}^{tw}$ and $Z(C_{\Gamma,\gamma}^{tw})$, there is no loss to assume $\gamma=0$ and thus this lemma follows from \cite[(6A)]{FS89}.
\end{proof}

As for defect groups in \cite[\S7]{FS89}, for any radical subgroup $\tR$, we have $\tN(\ttheta)=\group{\tau,N(\ttheta)}$ and
\[N(\ttheta) = N(\ttheta_0) \cap N(\theta_+), \quad N(\theta_+) \leq N(\zeta_+),\]
where $\zeta_+$ is the linear character of $Z(C_+)$ induced by $\ttheta_0$ and $\theta_+$.
Similarly as in \cite[(7B)]{FS89}, it is easy to see from Equation (\ref{equ:theta_+R_+}) that
\[N(\theta_+) = I_0(V_0) \times \prod_{\Gamma,\delta,i} N_{\Gamma,\delta,i}(\theta_{\Gamma,\delta,i})\wr\fS(t_{\Gamma,\delta,i}).\]
Note that the notation here is slightly different from the one in \cite{FS89}; in particular, $\gamma$ in \cite{FS89} and here have different meanings.
For $g \in N$, the function $\omega_g$ on $\tC_0$ is defined in the same way as in \cite[(7.5)]{FS89}.
Then the conclusions of \cite[(7D), (7E)]{FS89} for defect groups hold for any radical subgroups from Lemma \ref{lem:phi-Gamma}.
In particular, for any $g \in N(\theta_+)$, $\omega_g$ is a linear character of $\tC_0/C_0Z(C_+)$ of order $1$ or $2$.
\emph{In the sequel, when we refer to \cite[(7D), (7E)]{FS89}, we mean the generalized results for all radical subgroups.}
For conformal symplectic groups, \cite[(7F)]{FS89} can be simplified as follows.

\begin{lem}\label{lem:N(ttheta)}
Keep the above notation, then
\[N(\ttheta) = \set{ g \in N(\theta_+) \mid \ttheta_0\omega_g=\ttheta_0 }\]
and $N(\ttheta)$ is a normal subgroup of $N(\theta_+)$ of index $1$ or $2$.
\end{lem}

\begin{proof}
The proof is the same as for \cite[(7F)]{FS89}, noting that for symplectic groups, $g_0$ in the decomposition $g=g_0g_+$ satisfies $g_0\in\I_0(V_0)$ and thus ${^{g_0}\ttheta_0}=\ttheta_0$.
\end{proof}

\begin{prop}\label{prop:N(theta+):N(ttheta)}
Keep the above notation and let cases (I)$\sim$(VI) be as in Table \ref{tab:block}.
\[|N(\theta_+):N(\ttheta)|=
\left\{\begin{array}{ll}
1, & \textrm{cases (I)$\sim$(III), (V);}\\
2, & \textrm{cases (IV), (VI).}
\end{array}\right.\]
\end{prop}

Note that
\begin{equation*}\label{equ:tC_0}
\begin{aligned}
\tC_0/\I_0(V_0) &\cong \group{\tau_+,Z(C_+)},\\
\tC_0/Z(C_+) &\cong \group{\tau_0,\I_0(V_0)} = \J_0(V_0).
\end{aligned}
\addtocounter{thm}{1}\tag{\thethm}
\end{equation*}
Then similarly as on \cite[p.169]{FS89}, $\ttheta_0$ can be decomposed as $\ttheta_0 = \tchi_0\zeta$, where $\tchi_0$ is a character of $\J_0(V_0)$ and $\zeta$ is an extension to $\group{\tau_+,Z(C_+)}$ of the linear character $\zeta_+$ induced by $\theta_+$ and $\ttheta_0$.
Fix a choice $\zeta$ of the extension, then all possible choices are $\zeta\omega$ with $\omega \in \Irr(\J_0(V_0)/\I_0(V_0))$.

\begin{rem}\label{rem:zeta}
From now on, we fix an extension $\zeta$ of $\zeta_+$ to $\group{\tau_+,Z(C_+)}$, then $\tchi_0$ and $\ttheta_0$ determine each other since $\zeta$ is a linear character.
\end{rem}

\begin{proof}[Proof of Proposition \ref{prop:N(theta+):N(ttheta)}]
$s=i^*(\ts)$ can be decomposed as $s=s_0s_+$, where $s_0\in\I_0(V_0)$ and $s_+\in\I_0(V_+)$ can be determined by $\theta_0$ and $\theta_+$ respectively; see \cite[\S6.B]{Li19}.
Since $(\tR,\ttheta)$ is a $\tB$-Brauer pair and $\tB$ has label $(\ts,\cK)$, then by Proposition \ref{prop:Brauerpair} and the results in \cite[\S11]{FS89}, $\tchi_0$ can be chosen to be of the form $\tchi_0 = \chi_{\ts_0,\kappa}^{\J_0(V_0)}$, where $\ts_0\in\D_0(V_0^*)$ satisfies $s_0=i^*(\ts_0)$ and $\kappa\in\cK$.
By Lemma \ref{lem:N(ttheta)} and Remark \ref{rem:zeta}, $N(\ttheta) = \set{ g \in N(\theta_+) \mid \tchi_0=\tchi_0\omega_g }$.
By \cite[(7E)]{FS89}, $\omega_g\neq1$ for some $g \in N(\theta_+)$ if and only if $m_{X+1}(s_+)\neq0$.
For cases (I)$\sim$(III), we then have $\omega_g=1$ for all $g \in N(\theta_+)$, thus $|N(\theta_+):N(\ttheta)|=1$.
For cases (IV)$\sim$(VI), $\omega_g\neq1$ for some $g \in N(\theta_+)$, then the assertion follows from Remark \ref{rem:action-tz}.
\end{proof}

Recall that the $\tN$-conjugacy class of $\ttheta$ is the same as the $N$-conjugacy class of $\ttheta$.
The following is essentially in the spirit of \cite[\S11]{FS89}.

\begin{cor}\label{cor:cc-cano-chars}
Let $(\tR,\tvarphi)$ be a weight of $\tG$ belonging to the block $\tB$ with label $(\ts,\cK)$, where $\tvarphi$ lies over some $\ttheta\in\dz(\tC/Z(\tR))$ such that $(\tR,\ttheta)$ is a $\tB$-Brauer pair.
Let $\theta_+$ be as in (\ref{equ:theta_+R_+}).
Let $\Cl_{\tN}(\ttheta)$ be the $\tN$-conjugacy class of $\ttheta$.
Denote $\Cl_{N(\theta_+)}(\ttheta) = \set{ \ttheta_1\in \Cl_{\tN}(\ttheta) \mid \ttheta_1=\ttheta_{1,0}\theta_{1,+}, \theta_{1,+}=\theta_+ }$.
Then the set
\[ \Set{ \kappa \mid \ttheta_1=\ttheta_{1,0}\theta_+\in \Cl_{N(\theta_+)}(\ttheta), \ttheta_{1,0} \textrm{~induces~} \chi_{\ts_0,\kappa}^{\J_0(V_0)} } \]
coincides with $\cK$.
\end{cor}
\begin{proof}
By the proof of Proposition \ref{prop:N(theta+):N(ttheta)} and Table \ref{tab:block}.
\end{proof}

\begin{rem}\label{rem:cc-cano-chars}
By the remark at the beginning of this subsection, when we construct weight characters, we start with a representative $\ttheta$ of an $\tN$-conjugacy class of $\dz(\tC/Z(\tR))$.
We fix a way to choose such a representative.
Recall that $\ttheta=\ttheta_0\theta_+$ has two parts: the $0$-part $\ttheta_0$ and the $+$-part $\theta_+$.
Since $\Gamma$ determines an $N_\Gamma$-conjugacy class in $\dz(C_\Gamma/R_\Gamma)$, once we fix $\theta_\Gamma\in\dz(C_\Gamma/R_\Gamma)$ for each $\Gamma$, we fix the way to choose $\theta_+$; as in \cite[\S7]{FS89}, call it normalized.
For cases (I)$\sim$(III), (V) and any $\tN$-conjugacy class of $\dz(\tC/Z(\tR))$, by Proposition \ref{prop:N(theta+):N(ttheta)}, there is only one $\ttheta$ in it whose $+$-part $\theta_+$ is normalized; we take this $\ttheta$ as the representative of this $\tN$-conjugacy class.
For cases (IV), (VI) and any $\tN$-conjugacy class of $\dz(\tC/Z(\tR))$, by Proposition \ref{prop:N(theta+):N(ttheta)}, there are two $\ttheta$ in it whose $+$-part $\theta_+$ is normalized.
We can take either of the two choices as the representative of this $\tN$-conjugacy class, but once we fix one, we should always use that one.
\end{rem}

\begin{rem}\label{rem:ts0}
For the proof of Lemma \ref{lem:(3)of(iii)}, we make a convention related to Remark \ref{rem:action-tz} and Remark \ref{rem:zeta}.
Fixing an extension $\zeta$ as in Remark \ref{rem:zeta} gives a decomposition $\ts=\ts_0\ts_+$, where $\tchi_0 = \chi_{\ts_0,\kappa}^{\J_0(V_0)}$ as in the proof of Proposition \ref{prop:N(theta+):N(ttheta)}.
Let $s$, $\ts$ and the labels for characters in $\Irr(\tG)$ be as in Remark \ref{rem:action-tz}; but let $\tz_0$ be a generator of $Z(\tG^*)_{\ell'}$.
We start with $\tG$ of a fixed dimension $2n$, and consider blocks with labels $(\ts\tz_0^i,\cK)$, where $i \in \set{0,1,\ldots,(q-1)_{\ell'}-1}$ if $m_{X+1}(s)=0$ while $i \in \set{0,1,\ldots,(q-1)_{\ell'}/2-1}$ if $m_{X+1}(s)\neq0$.
Then we have a fixed symplectic space $V_0$ and the conformal symplectic group $\J_0(V_0)$ over $V_0$.
Thus we can fix particular $\zeta$'s for all blocks with labels $(\ts\tz_0^i,\cK)$ such that the induced characters of $\J_0(V_0)$ are $\chi_{\ts_0\tz_0^i,\kappa}^{\J_0(V_0)}$'s.
\end{rem}

\subsection{}\label{subsec:weight}
In this subsection, we label the weights of $\tG$ and prove the blockwise Alperin weight conjecture for $\tG$.
We start with a lemma about the construction of characters under some \emph{ad hoc} conditions.

\begin{lem}\label{lem:cons-char}
Let $K= K_0 \times_Z K_+$ be a central product over $Z= K_0 \cap K_+$.
Assume $\ttheta=\ttheta_0\theta_+$ is a character of $K$.
Denote by $\zeta_+$ the linear character of $Z$ induced by $\ttheta_0$ and $\theta_+$.
Let $H=K_0H_+$ be such that $K_0 \cap H_+ =Z$, $K_0 \unlhd H$, $K_+ \unlhd H_+ \unlhd H$.
Assume $H_+$ stabilizes $\ttheta_0$ and $\theta_+$, and assume $[K_0,H_+] \leq \Ker\zeta_+$.
Let $\psi_+ \in \Irr(H_+\mid\theta_+)$ and define $\psi:=\ttheta_0\psi_+$ by $\psi(k_0h_+)=\ttheta_0(k_0)\psi_+(h_+)$ for any $k_0 \in K_0$ and $h_+ \in H_+$.
Then $\psi$ is an irreducible character of $H$.
\end{lem}
\begin{proof}
By the assumption, we have $\Irr(Z\mid\psi_+) = \set{\zeta_+}$.
Let $\rho_0: K_0 \to \GL(V_0)$, $\rho_+: H_+ \to \GL(V_+)$ be representations affording $\ttheta_0, \psi_+$ respectively.
Define $\rho: H \to \GL(V_0 \otimes V_+), k_0h_+ \mapsto \rho_0(k_0)\otimes\rho_+(h_+)$.
We first check that this map is well defined.
Assume $k_0h_+=k'_0h'_+$, then $k_0^{-1}k'_0=h_+{h'_+}^{-1} \in Z$ and
\begin{align*}
\rho(k'_0h'_+) &= \rho_0(k'_0) \otimes \rho_+(h'_+) = \rho_0 (k_0k_0^{-1}k'_0) \otimes \rho_+(h'_+) \\
&= \rho_0(k_0)\zeta_+(k_0^{-1}k'_0) \otimes \rho_+(h'_+) = \rho_0(k_0) \otimes \zeta_+(h_+{h'_+}^{-1})\rho_+(h'_+) \\
&= \rho_0(k_0) \otimes \rho_+(h_+{h'_+}^{-1}h'_+) = \rho_0(k_0) \otimes \rho_+(h_+) = \rho(k_0h_+).
\end{align*}
To show that $\rho$ is multiplicative, note that $k_0h_+k'_0h'_+=k_0k'_0[{k'_0}^{-1},h_+]h_+h'_+$, so
\begin{align*}
\rho(k_0h_+k'_0h'_+) &= \rho_0(k_0k'_0[{k'_0}^{-1},h_+]) \otimes \rho_+(h_+h'_+) \\
&= \rho_0(k_0k'_0)\zeta_+([{k'_0}^{-1},h_+]) \otimes \rho_+(h_+h'_+) \\
&= \rho_0(k_0)\rho_0(k'_0) \otimes \rho_+(h_+)\rho_+(h'_+) \\
&= \big(\rho_0(k_0) \otimes \rho_+(h_+)\big) \big(\rho_0(k'_0) \otimes \rho_+(h'_+)\big) \\
&= \rho(k_0h_+)\rho(k'_0h'_+).
\end{align*}
Thus $\rho$ is a representation of $H$.
To see that $\rho$ is irreducible, it suffices to note that $\rho$ can be lifted to $K_0 \times H_+$ via the surjective homomorphism $K_0 \times H_+ \to H, (k_0,h_+) \mapsto k_0h_+$.
\end{proof}

Set $\sC_{\Gamma,\delta} = \cup_{i} \dz(N_{\Gamma,\delta,i}(\theta_{\Gamma,\delta,i})/R_{\Gamma,\delta,i}\mid\theta_{\Gamma,\delta,i})$ and assume $\sC_{\Gamma,\delta}=\set{\psi_{\Gamma,\delta,i,j}}$ with $\psi_{\Gamma,\delta,i,j}$ a character of $N_{\Gamma,\delta,i}(\theta_{\Gamma,\delta,i})$.
Note that $\sC_{\Gamma,\delta}$ here is different from but in bijection with that in \cite[\S6.B]{Li19}.
Here, we define $\sC_{\Gamma,\delta}$ as above so that we can use Lemma \ref{lem:cons-char} to construct $\dz(\tN(\ttheta)/\tR \mid \ttheta)$ from $\dz(N_+(\theta_+)/R_+ \mid \theta_+)$.

Recall that $N:=N_G(R)$ can be decomposed as $N = N_0 \times N_+$.
From Equation (\ref{equ:theta_+R_+}), we have $N_+(\theta_+) = \prod_{\Gamma,\delta,i} N_{\Gamma,\delta,i}(\theta_{\Gamma,\delta,i})\wr\fS(t_{\Gamma,\delta,i})$.
Any $\psi_+ \in \dz(N_+(\theta_+)/R_+ \mid \theta_+)$ can be decomposed as $\psi_+ = \prod_{\Gamma,\delta,i} \psi_{\Gamma,\delta,i}$, where $\psi_{\Gamma,\delta,i}$ is a character of $N_{\Gamma,\delta,i}(\theta_{\Gamma,\delta,i})\wr\fS(t_{\Gamma,\delta,i})$.
Then by Clifford theory, $\psi_{\Gamma,\delta,i}$ is of the form
\begin{equation}\label{equ:psi}
\Ind_{N_{\Gamma,\delta,i}(\theta_{\Gamma,\delta,i})\wr\prod_j\fS(t_{\Gamma,\delta,i,j})}^{N_{\Gamma,\delta,i}(\theta_{\Gamma,\delta,i})\wr\fS(t_{\Gamma,\delta,i})}
\overline{\prod_j\psi_{\Gamma,\delta,i,j}^{t_{\Gamma,\delta,i,j}}} \cdot \prod_j\phi_{\kappa_{\Gamma,\delta,i,j}},
\addtocounter{thm}{1}\tag{\thethm} 
\end{equation}
where $t_{\Gamma,\delta,i}=\sum_jt_{\Gamma,\delta,i,j}$, $\overline{\prod_j\psi_{\Gamma,\delta,i,j}^{t_{\Gamma,\delta,i,j}}}$ is the canonical extension of $\prod_j\psi_{\Gamma,\delta,i,j}^{t_{\Gamma,\delta,i,j}}\in\Irr(N_{\Gamma,\delta,i}^{t_{\Gamma,\delta,i}})$ to $N_{\Gamma,\delta,i}\wr\prod_j\fS(t_{\Gamma,\delta,i,j})$ as in the proof of \cite[Proposition~2.3.1]{Bon99b}, $\kappa_{\Gamma,\delta,i,j}\vdash t_{\Gamma,\delta,i,j}$ without $\ell$-hook and $\phi_{\kappa_{\Gamma,\delta,i,j}}$ is the character of $\fS(t_{\Gamma,\delta,i,j})$ corresponding to $\kappa_{\Gamma,\delta,i,j}$.
Now, define $K_\Gamma:\cup_\delta\sC_{\Gamma,\delta}\to\{\ell\textrm{-cores}\}$, $\psi_{\Gamma,\delta,i,j}\mapsto\kappa_{\Gamma,\delta,i,j}$ and $K= \prod_\Gamma K_\Gamma$.
We call $K$ the label of $\psi_+$, which will be denoted as $\psi_{+,K}$ from now on.

As before, the block $\tB$ of $\tG$ with label $(\ts,\cK)$ is denoted by $\tB_{\ts,\cK}$.
Denote by $i\cW_{\ts,\cK}^0$ the set of $K$'s satisfying $\sum_{\delta,i,j}\ell^\delta |K_\Gamma(\psi_{\Gamma,\delta,i,j})|=w_\Gamma$ and  $m_\Gamma(s)=m_\Gamma(s_0)+\beta_\Gamma e_\Gamma w_\Gamma$, where $\beta_\Gamma=1$ or $2$ according to $\Gamma\notin\cF_0$ or $\Gamma\in\cF_0$.
Then by the above constructions, $i\cW_{\ts,\cK}^0$ is in bijection with the set $\cup_{\set{(\tR,\ttheta)\in\tB}/\sim\tN}\dz(N_+(\theta_+)/R_+ \mid \theta_+)$, where $(\tR,\ttheta)$ runs over an $\tN$-transversal of Brauer pairs $(\tR,\ttheta)$ belonging to the block $\tB_{\ts,\cK}$.
Denote this bijection as $K \mapsto \psi_{+,K}$.
Recall that for any $K \in i\cW_{\ts,\cK}^0$, $K'$ is defined as in the paragraph before \cite[Proposition 6.25]{Li19}.

\begin{lem}\label{lem:alpha-(IV)}
Assume we are in case (IV) or (VI) in Table \ref{tab:block} and $\alpha\in\Irr\left(N_+(\theta_+)/N_+(\ttheta)\right)$ is the unique non-trivial character, then $\alpha\cdot\psi_{+,K}=\psi_{+,K'}$.
\end{lem}
\begin{proof}
Recall that $N_+(\theta_+) = \prod_{\Gamma,\delta,i} N_{\Gamma,\delta,i}(\theta_{\Gamma,\delta,i})\wr\fS(t_{\Gamma,\delta,i})$.
By \cite[(7E)]{FS89} and Lemma \ref{lem:N(ttheta)}, $N_{\Gamma,\delta,i}(\theta_{\Gamma,\delta,i})\wr\fS(t_{\Gamma,\delta,i}) \leq N_+(\ttheta)$ for $\Gamma \neq X+1$.
So there is no loss to assume that $N_+(\theta_+) = \prod_{\delta,i} N_{X+1,\delta,i}\wr\fS(t_{\delta,i})$; note that $N_{X+1,\delta,i}(\theta_{X+1,\delta,i}) = N_{X+1,\delta,i}$.
Since $N_{X+1,\delta,i}\wr\fS(t_{\delta,i}) \nleq N_+(\ttheta)$ for each component, we may assume there is only one component in the radical subgroup $R_+ = R_{X+1,\delta,i}^t$.
Thus $N_+(\theta_+) = N_{X+1,\delta,i}\wr\fS(t)$ and
\begin{equation*}
\psi_+ = \Ind_{N_{X+1,\delta,i}\wr\prod_j\fS(t_j)}^{N_{X+1,\delta,i}\wr\fS(t)}
\overline{\prod_j\psi_{X+1,\delta,i,j}^{t_j}} \cdot \prod_j\phi_{\kappa_{X+1,\delta,i,j}}.
\end{equation*}
So
\begin{equation*}
\alpha\cdot\psi_+ = \Ind_{N_{X+1,\delta,i}\wr\prod_j\fS(t_j)}^{N_{X+1,\delta,i}\wr\fS(t)}
\overline{\prod_j(\Res^{N_+(\theta_+)}_{N_{X+1,\delta,i}}\alpha\cdot\psi_{X+1,\delta,i,j})^{t_j}} \cdot \prod_j\phi_{\kappa_{X+1,\delta,i,j}}.
\end{equation*}
By \cite[(7E)]{FS89} again, $\Res^{N_+(\theta_+)}_{N_{X+1,\delta,i}}\alpha \neq 1$ and
\[\Res^{N_+(\theta_+)}_{N_{X+1,\delta,i}}\alpha\cdot\psi_{X+1,\delta,i,j} = \psi_{X+1,\delta,i,j'},\]
where $\psi_{X+1,\delta,i,j'}$ is as in \cite[Lemma 6.24]{Li19}.
Thus $\alpha\cdot\psi_{+,K}=\psi_{+,K'}$ by the definition of $K'$.
\end{proof}

Now, set $i\cW_{\ts,\cK}=i\cW_{\ts,\cK}^0$ for cases (I)$\sim$(III) and (V), and set 
\[i\cW_{\ts,\cK} = \Set{ \set{K,K'} \mid K\neq K' \in i\cW_{\ts,\cK}^0 } \cup \Set{ (K,i) \mid K=K' \in i\cW_{\ts,\cK}^0, i\in \Z/2\Z }\]
for cases (IV) and (VI).

\begin{prop}\label{prop:weights}
Assume $\tB$ is a block of $\tG$ with label $(\ts,\cK)$, then $\Alp(\tB)$ is in bijection with $i\cW_{\ts,\cK}$.
\end{prop}

\begin{proof}
Recall that $i\cW_{\ts,\cK}^0$ is in bijection with $\cup_{\set{(\tR,\ttheta)\in\tB}/\sim\tN}\dz(N_+(\theta_+)/R_+ \mid \theta_+)$.

(1) Assume we are in one of the cases (I)$\sim$(III).
Then by Proposition \ref{prop:N(theta+):N(ttheta)} and its proof, $N_+(\theta_+) = N_+(\ttheta_0) \cap N_+(\theta_+) = N_+(\ttheta)$ and $\omega_g=1$ for any $g \in N_+(\theta_+)$.
By the definition of $\omega_g$ (\cite[(7.5)]{FS89}), $[\tC_0,N_+(\theta_+)] \leq \Ker\zeta_+$, where $\zeta_+$ is the linear character of $Z(C_+)$ induced by $\ttheta_0$ and $\theta_+$.
Then by applying Lemma \ref{lem:cons-char} to $\tC=\tC_0C_+$ and $\tN(\ttheta) = \tC_0N_+(\theta_+)$, $K \mapsto \psi_K := \ttheta\psi_{+,K}$ gives a bijection from $i\cW_{\ts,\cK}^0$ to $\cup_{\set{(\tR,\ttheta)\in\tB}/\sim\tN}\dz(\tN(\ttheta)/\tR \mid \ttheta)$.
By Clifford theory, the latter set is in bijection with the set of $\tB$-weights via $\psi_K \mapsto \Ind_{\tN(\ttheta)}^{\tN} \psi_K$ and thus the assertion holds in this case.

(2) Assume we are in case (V).
Then by Proposition \ref{prop:N(theta+):N(ttheta)} and its proof, $N_+(\theta_+) = N_+(\ttheta_0) \cap N_+(\theta_+) = N_+(\ttheta)$ but $\omega_g\neq1$ for some $g \in N_+(\theta_+)$.
We can not apply Lemma \ref{lem:cons-char} directly.
Let $\hC_0 = \group{\tau^2,C_0Z(C_+)}$.
As in the proof of Proposition \ref{prop:N(theta+):N(ttheta)}, $\ttheta_0 = \tchi_0 \zeta$ with $\tchi_0 = \chi_{\ts_0,\kappa}^{\J_0(V_0)}$, where $\cK=\set{\kappa}$.
Then $\Res^{\tC_0}_{\hC_0} \ttheta_0$ has two irreducible constituents since $\Res^{\tC_0}_{C_0} \ttheta_0 = \Res^{\J_0(V_0)}_{\I_0(V_0)} \tchi_0$ has two irreducible constituents by \cite[Theorem 3.2]{Li19}.
Let $\htheta_0^{(0)}$ and $\htheta_0^{(1)}$ be the two irreducible constituents of $\Res^{\tC_0}_{\hC_0} \ttheta_0$.
By \cite[(7E)]{FS89}, $\Res^{\tC_0}_{\hC_0}\omega_g=1$, which means $[\hC_0,N_+(\theta_+)] \leq \Ker\zeta_+$ by definition of $\omega_g$.
Then by applying Lemma \ref{lem:cons-char} to $\hC:=\hC_0C_+$ and $\hN(\ttheta) := \hC_0N_+(\theta_+)$, there is a character $\hpsi_K := \htheta_0^{(0)}\psi_{+,K}\in \dz(\hN(\ttheta)/\tR)$.
Then $\psi_K = \Ind_{\hN(\ttheta)}^{\tN(\ttheta)} \hpsi_K \in \dz(\tN(\ttheta)/\tR \mid \ttheta)$.
If we start with $\htheta_0^{(1)}$, we get the same $\psi_K$.
This gives a bijection from $i\cW_{\ts,\cK}^0$ to $\cup_{\set{(\tR,\ttheta)\in\tB}/\sim\tN}\dz(\tN(\ttheta)/\tR \mid \ttheta)$, and proves the assertion in this case by Clifford theory again.

(3) Assume we are in case (IV) or (VI).
Then $|N_+(\theta_+):N_+(\ttheta)|=2$ and $\omega_g=1$ for any $g \in N_+(\ttheta)$, thus $[\tC_0,N_+(\ttheta)] \leq \Ker\zeta_+$.
This case is divided into two sub-cases.

(3.1) Assume $K \neq K'$.
By Lemma \ref{lem:alpha-(IV)} and Clifford theory, $\psi_{+,K}^0:=\Res^{N_+(\theta_+)}_{N_+(\ttheta)} \psi_{+,K}$ is irreducible.
Applying Lemma \ref{lem:cons-char} to $\tC=\tC_0C_+$, $\tN(\ttheta)=\tC_0N_+(\ttheta)$ and $\psi_{+,K}^0$, there is an irreducible character $\psi_K^0 := \ttheta_0\psi_{+,K}^0$ of $\tN(\ttheta)$.
Then $\Ind_{\tN(\ttheta)}^{\tN}\psi_K^0$ gives a weight associated to $\tR$.
If we start with $\psi_{+,K'}$, we will get the same weight, since $\Res^{N_+(\theta_+)}_{N_+(\ttheta)} \psi_{+,K} = \Res^{N_+(\theta_+)}_{N_+(\ttheta)} \psi_{+,K'}$.
Give this weight the label $\set{K,K'}$.

(3.2) Assume $K=K'$.
By Lemma \ref{lem:alpha-(IV)} and Clifford theory, $\Res^{N_+(\theta_+)}_{N_+(\ttheta)} \psi_{+,K}$ has two irreducible constituents, denoted as $\psi_{+,K}^{(i)}$, $i\in\Z/2\Z$.
Applying Lemma \ref{lem:cons-char} to $\tC=\tC_0C_+$, $\tN(\ttheta)=\tC_0N_+(\ttheta)$ and $\psi_{+,K}^{(i)}$, there is an irreducible character $\psi_K^{(i)} := \ttheta_0\psi_{+,K}^{(i)}$ of $\tN(\ttheta)$.
Then $\Ind_{\tN(\ttheta)}^{\tN}\psi_K^{(i)}$ gives a weight associated to $\tR$.
Give this weight the label $(K,i)$.

Combining (3.1) and (3.2), we have a bijection from $i\cW_{\ts,\cK}$ to $\Alp(\tB)$, which proves the assertion in this case.
\end{proof}

Denote the conjugacy class of weights corresponding by the above Proposition to $K \in i\cW_{\ts,\cK}$ for cases (I)$\sim$(III) and (V) by $w_K$ or $w_{\ts,\cK,K}$; denote the conjugacy class of weights corresponding to $\set{K,K'} \in i\cW_{\ts,\cK}$ for cases (IV) and (VI) by $w_{\set{K,K'}}$ or $w_{\ts,\cK,\set{K,K'}}$ and the conjugacy class of weights corresponding to $(K,i) \in i\cW_{\ts,\cK}$ for cases (IV) and (VI) by $w_{(K,i)}$ or $w_{\ts,\cK,(K,i)}$.

Recall that $|\sC_{\Gamma,\delta}| = \beta_\Gamma e_\Gamma\ell^\delta$ by \cite[(4A)]{An94}.
By \cite[(1A)]{AF90}, there is a bijection between $i\cW_{\ts,\cK}^0$ and the following set
\[i\cW_{\ts,\cK}^{q,0} = \Set{ Q = \prod_\Gamma Q_\Gamma ~\middle|~
\begin{array}{c} Q_\Gamma=(Q_\Gamma^{(1)},\cdots,Q_\Gamma^{(\beta_\Gamma e_\Gamma)}),\\
Q_\Gamma\textrm{'s are partitions, } \sum_{i=1}^{\beta_\Gamma e_\Gamma} |Q_\Gamma^{(i)}| = w_\Gamma. \end{array}}\]
For $Q_{X+1} = (\lambda_1,\dots,\lambda_e,\mu_1,\dots,\mu_e)$, define $(Q_{X+1})' = (\mu_1,\dots,\mu_e,\lambda_1,\dots,\lambda_e)$, and define $Q'$ by $Q'_\Gamma = Q_\Gamma$ for $\Gamma \neq X+1$ and $Q'_{X+1} = (Q_{X+1})'$.
So if $K$ corresponds to $Q$, $K'$ corresponds to $Q'$.
Similarly as before, set $i\cW_{\ts,\cK}^q=i\cW_{\ts,\cK}^{q,0}$ for cases (I)$\sim$(III) and (V), and set 
\[i\cW_{\ts,\cK}^q = \Set{ \set{Q,Q'} \mid Q\neq Q' \in i\cW_{\ts,\cK}^{q,0} } \cup \Set{ (Q,i) \mid Q=Q' \in i\cW_{\ts,\cK}^{q,0}, i\in \Z/2\Z }\]
for cases (IV) and (VI).
Then $i\cW_{\ts,\cK}^q$ is in bijection with $i\cW_{\ts,\cK}$ and thus with $\cW(\tB_{\ts,\cK})$.
Denote the conjugacy class of weights corresponding to $Q \in i\cW_{\ts,\cK}^q$ for cases (I)$\sim$(III) and (V) by $w_Q$ or $w_{\ts,\cK,Q}$; denote the conjugacy class of weights corresponding to $\set{Q,Q'} \in i\cW_{\ts,\cK}^q$ for cases (IV) and (VI) by $w_{\set{Q,Q'}}$ or $w_{\ts,\cK,\set{Q,Q'}}$ and the conjugacy class of weights corresponding to $(Q,i) \in i\cW_{\ts,\cK}^q$ for cases (IV) and (VI) by $w_{(Q,i)}$ or $w_{\ts,\cK,(Q,i)}$.

\paragraph{}
Now, Theorem \ref{mainthm-1} follows from the following explicit bijection.

\begin{thm}\label{thm:bijection}
Let $\tB=\tB_{\ts,\cK}$ be a block of $\tG$.
Keep the notation above and that of Lemma \ref{lem:IBr}.
\begin{enumerate}[(1)]\setlength{\itemsep}{0pt}
\item Assume we are in one of the cases (I)$\sim$(III) and (V) with $\cK=\set{\kappa}$.
	Then there is a bijection
	\[\tOmega_{\tB}: \quad \IBr(\tB_{\ts,\set{\kappa}}) \longleftrightarrow \Alp(\tB_{\ts,\set{\kappa}})\]
	satisfying that if $\tOmega_{\tB}(\phi_{\ts,\lambda}) = w_{\ts,\cK,Q}$,
	then $\lambda_\Gamma=\kappa_\Gamma*Q_\Gamma$.
	Note that in cases (II) or (III), then $\lambda_{X+1}=\kappa_{X+1}$.
\item Assume we are in case (IV) with $\cK=\set{\kappa}$ and $\kappa_{X+1}=\O$
	or in case (VI) with $\cK=\set{\kappa,\kappa'}$ and $\kappa_{X+1}\neq\O$ degenerate.
	Then there is a bijection
	\[\tOmega_{\tB}: \quad \IBr(\tB_{\ts,\cK}) \longleftrightarrow \Alp(\tB_{\ts,\cK})\]
	satisfying that if $\tOmega_{\tB}(\phi_{\ts,\lambda})=w_{\ts,\cK,\set{Q,Q'}}$ or $w_{\ts,\cK,(Q,i)}$, then
	\begin{enumerate}[(i)]\setlength{\itemsep}{0pt}
	\item $\lambda_\Gamma=\kappa_\Gamma*Q_\Gamma$ for $\Gamma \neq X+1$;
	\item $\lambda_{X+1}=\kappa_{X+1}*Q_\Gamma$ if $Q \neq Q'$;
	\item $\lambda_{X+1}=\kappa_{X+1}*(Q,i)$ if $Q=Q'$.
	\end{enumerate}
\end{enumerate}
\end{thm}

\begin{proof}
The theorem follows obviously from the \S\ref{subsec:partition-symbol}.
\end{proof}


\section{The inductive condition}\label{sec:inductive}

In this section, we verify the inductive condition for the cases in Theorem \ref{mainthm-2}.
Let $E$ be the set of field automorphisms on $\tG$ and $\cB$ be as in Table \ref{tab:block}.
Then $(\tG E)_B\leq(\tG E)_\cB$ holds obviously for any $B \in \tB$.
We first remark that the part (i) of Theorem \ref{thm:criterion} holds.
The first and the second requirements are obvious.
The rest follows from the fact that $\tG/G$ is cyclic.

\subsection{}\label{subsec:pro-bij}
In this subsection, we prove part (ii) of Theorem \ref{thm:criterion}.
Set $\tOmega^{\tG} = \cup_{\tB\in\Bl(\tG)} \tOmega_{\tB}$, where $\tOmega_{\tB}$ is the bijection in Theorem \ref{thm:bijection}

\begin{lem}\label{lem:equi}
Assume the decomposition matrix with respect to $\cE(\tG,\ell')$ is unitriangular.
Then the bijection $\tOmega^{\tG}$ is $\Aut(G)$-equivariant.
\end{lem}
\begin{proof}
It suffices to consider a field automorphism $\sigma\in E$.

Let $\chi_{\ts,\lambda} \in \Irr(\tG)$, then by \cite[Theorem 3.1]{CS13}, $\chi_{\ts,\lambda}^\sigma = \tchi_{\sigma^*(\ts),\sigma^*(\lambda)}$, where $\sigma^*$ is the corresponding field automorphism on $\tG^*$ via duality as in \cite[Definition 2.1]{CS13} and $\sigma^*(\lambda)_{\sigma^*(\Gamma)} = \lambda_\Gamma$; see also \cite[Theorem 5.2]{Li19} and the proof of \cite[Proposition 5.6]{Li19}.
By the assumption on the decomposition matrices and \cite[Lemma 7.5]{CS13}, the labels of irreducible Brauer characters in Lemma \ref{lem:IBr} can be chosen to satisfy that $\phi_{\ts,\lambda}^\sigma = \phi_{\sigma^*(\ts),\sigma^*(\lambda)}$.
Let $\tB_{\ts,\cK}$ be a block of $\tG$, then $\tB_{\ts,\cK}^\sigma = \tB_{\sigma^*(\ts),\sigma^*(\cK)}$, where $\sigma^*(\cK)$ is similarly defined.

Let $(\tR,\tvarphi)$ be a $\tB_{\ts,\cK}$-weight; we use the notation in the proof of Proposition \ref{prop:weights}.
Assume $(\tR,\tvarphi)$ has label $(\ts,\cK,K)$ or $(\ts,\cK,\set{K,K'})$.
Then $(\tR,\varphi)^\sigma$ is a $\tB_{\sigma^*(\ts),\sigma^*(\cK)}$-weight.
Assume $\tvarphi$ comes from a Brauer pair $(\tR,\ttheta)$ and is constructed from $\psi_{+,K}\in\dz(N_+(\theta_+)/R_+ \mid \theta_+)$ as in Proposition \ref{prop:weights}.
By \cite[Proposition 6.20]{Li19}, $\psi_{+,K}^\sigma = \psi_{+,\sigma^*(K)}$ up to $N_+$-conjugate, where $\sigma^*(K)$ is defined similarly as before.
Thus $(\tR,\varphi)^\sigma$ has label $(\sigma^*(\ts),\sigma^*(\cK),\sigma^*(K))$ or $(\sigma^*(\ts),\sigma^*(\cK),\set{\sigma^*(K),\sigma^*(K)'})$.
Assume $(\tR,\tvarphi)$ has label $(\ts,\cK,(K,i))$ with $K=K'$.
We claim that $(\psi_{+,K}^{(i)})^\sigma = \psi_{+,\sigma^*(K)}^{(i)}$ up to $N_+$-conjugate, whose proof is left to Lemma \ref{lem:equi-claim}.
Thus $(\tR,\varphi)^\sigma$ has label $(\sigma^*(\ts),\sigma^*(\cK),(\sigma^*(K),i))$.

Now, it is easy to see that the bijection in Theorem \ref{thm:bijection} is equivariant.
\end{proof}

\begin{lem}\label{lem:equi-claim}
With the notation in the proof of the above lemma, $(\psi_{+,K}^{(i)})^\sigma = \psi_{+,\sigma^*(K)}^{(i)}$ up to $N_+$-conjugate.
\end{lem}
\begin{proof}
To avoid constant use of ``up to conjugation'', we transfer to twisted groups.
As in the proof of Lemma \ref{lem:alpha-(IV)}, we may assume $R_+^{tw} = (R_{X+1,\gamma,\bc}^{tw})^t$.
Thus $N_+^{tw}(\theta_+^{tw}) = N_{X+1,\gamma,\bc}^{tw}\wr\fS(t)$ and
\begin{equation*}
\psi_{+,K}^{tw} = \Ind_{N_{X+1,\gamma,\bc}^{tw}\wr\prod_j\fS(t_j)}^{N_{X+1,\gamma,\bc}^{tw}\wr\fS(t)}
\overline{\prod_j(\psi_{X+1,j}^{tw})^{t_j}} \cdot \prod_j\phi_{\kappa_{X+1,j}}.
\end{equation*}
Then $K$ is defined as $\psi_{X+1,j} \mapsto \kappa_{X+1,j}$.
Since $N_{X+1,\gamma,\bc}^{tw}/R_{X+1,\gamma,\bc}^{tw} \cong N_{X+1,\gamma}^{tw}/R_{X+1,\gamma}^{tw} \times N_{\fS(\ell^{|\bc|})}(A_{\bc})/A_{\bc}$, we may assume $\bc=\zero$.
By \cite[Lemma 6.15]{Li19}, we may assume $\gamma=0$.

Note that $|\dz(N_{X+1}/R_{X+1}\mid\theta_{X+1})| = 2e$ and by \cite[Lemma 6.23, Lemma 6.24]{Li19} the characters in $\dz(N_{X+1}/R_{X+1}\mid\theta_{X+1})$ can be labelled as $\psi_{X+1,1}, \psi_{X+1,1'}, \cdots, \psi_{X+1,e}, \psi_{X+1,e'}$.
Since $K=K'$, $\kappa_{X+1,j}=\kappa_{X+1,j'}$ by the definition of $K'$; in particular, $t_j=t_{j'}$.
Thus
\begin{equation*}
\psi_{+,K}^{tw} = \Ind_{N_{X+1}^{tw}\wr\prod_{j=1}^e(\fS(t_j)\times\fS(t_j))}^{N_{X+1}^{tw}\wr\fS(t)}
\overline{\prod_{j=1}^e(\psi_{X+1,j}^{tw})^{t_j}(\psi_{X+1,j'}^{tw})^{t_j}} \cdot \prod_j\phi_{\kappa_{X+1,j}}\phi_{\kappa_{X+1,j}}.
\end{equation*}
Let $T,V$ be as in the proof of \cite[Lemma 6.23]{Li19}, thus $N_{X+1}^{tw}=TV$.
Set $\bar{N}_{X+1}^{tw} = N_{X+1}^{tw}/T$ and use the bar convention.
Set $B_0 = \Set{ (n_1,\cdots,n_t) \mid n_i \in N_{X+1}^{tw}, \bar{n}_1 \cdots \bar{n}_t \in \bar{V}^2 }$.
By Lemma \ref{lem:N(ttheta)} and \cite[(7E)]{FS89}, $N_+(\ttheta) = B_0 \rtimes \fS(t)$.
To prove this lemma, it is equivalent to prove a similar assertion for the following character
\begin{equation*}
\Ind_{N_{X+1}^{tw}\wr\prod_{j=1}^e(\fS(t_j)\times\fS(t_j))}^{N_{X+1}^{tw}\wr\fS(t)}
\overline{\prod_{j=1}^e(\psi_{X+1,j}^{tw})^{t_j}(\psi_{X+1,j'}^{tw})^{t_j}},
\end{equation*}
which, for convenience, is still denoted by $\psi_{+,K}$.
Then
\begin{equation*}
\Res^{N_+^{tw}(\theta_+)}_{N_+^{tw}(\ttheta)} \psi_{+,K}^{tw} = \Ind_{B_0\rtimes\prod_{j=1}^e(\fS(t_j)\times\fS(t_j))}^{B_0\rtimes\fS(t)}
\Res^{N_{X+1}^{tw}\wr\prod_{j=1}^e(\fS(t_j)\times\fS(t_j))}_{B_0\rtimes\prod_{j=1}^e(\fS(t_j)\times\fS(t_j))}
\overline{\prod_{j=1}^e(\psi_{X+1,j}^{tw})^{t_j}(\psi_{X+1,j'}^{tw})^{t_j}}.
\end{equation*}
Set $N_{+,0}^{tw}(\ttheta) = \left(B_0\rtimes\prod_{j=1}^e\fS(t_j)\times\fS(t_j)\right)\rtimes\fS(2)$, where the action of $\fS(2)$ transposes each pair $(\psi_{X+1,j}^{tw},\psi_{X+1,j'}^{tw})$.
Note that $\psi_{X+1,j'}^{tw}=\alpha\psi_{X+1,j}^{tw}$ by the proof of \cite[Lemma 6.23]{Li19}, where $\alpha$ is the unique character of $\bar{V}$ of order $2$.
Then by the definition of $B_0$, $N_{+,0}^{tw}(\ttheta)$ fixes
\begin{equation*}
\Res^{N_{X+1}^{tw}\wr\prod_{j=1}^e(\fS(t_j)\times\fS(t_j))}_{B_0\rtimes\prod_{j=1}^e(\fS(t_j)\times\fS(t_j))}
\overline{\prod_{j=1}^e(\psi_{X+1,j}^{tw})^{t_j}(\psi_{X+1,j'}^{tw})^{t_j}}.
\end{equation*}
Thus this character has two extensions to $N_{+,0}^{tw}(\ttheta)$, denoted $\psi_{+,K,0}^{(0)}$ and $\psi_{+,K,0}^{(1)}$.
So $\Ind_{N_{+,0}^{tw}(\ttheta)}^{N_+^{tw}(\ttheta)}\psi_{+,K,0}^{(i)}$ are the two irreducible constituents of $\Res^{N_+^{tw}(\theta_+)}_{N_+^{tw}(\ttheta)} \psi_{+,K}^{tw}$; we may set $\psi_{+,K}^{(i),tw} = \Ind_{N_{+,0}^{tw}(\ttheta)}^{N_+^{tw}(\ttheta)}\psi_{+,K,0}^{(i)}$, where $\psi_{+,K}^{(i),tw}$ is the twisted version of $\psi_{+,K}^{(i)}$.

By \cite[Lemma 6.17]{Li19}, $\sigma$ fixes 
\begin{equation*}
\Res^{N_{X+1}^{tw}\wr\prod_{j=1}^e(\fS(t_j)\times\fS(t_j))}_{B_0\rtimes\prod_{j=1}^e(\fS(t_j)\times\fS(t_j))}
\overline{\prod_{j=1}^e(\psi_{X+1,j}^{tw})^{t_j}(\psi_{X+1,j'}^{tw})^{t_j}}.
\end{equation*}
Since the above character is a linear character, $\sigma$ fixes both the extensions $\psi_{+,K,0}^{(i)}$, and thus fixes $\psi_{+,K}^{(i),tw}$, which means that $(\psi_{+,K}^{(i)})^\sigma = \psi_{+,\sigma^*(K)}^{(i)}$ up to $N_+$-conjugacy.
\end{proof}

Recall that $\tcB = \Bl(\tG\mid\cB)$.
Set $\IBr(\tcB) = \cup_{\tB\in\tcB} \IBr(\tB)$ and $\Alp(\tcB) = \cup_{\tB\in\tcB} \Alp(\tB)$.
Denote by $\tOmega_{\tcB}$ the restriction of $\tOmega^{\tG}$ to $\IBr(\tcB)$.
As a corollary of Lemma \ref{lem:equi}, $\tOmega_{\tcB}$ is $E_{\cB}$-equivariant.
Since the bijection $\tOmega^{\tG}$ comes from the blockwise bijections $\tOmega_{\tB}$, $\tOmega_{\tcB}(\IBr(\tB)) = \Alp(\tB)$ for every $\tB \in \tcB$.

\begin{lem}\label{lem:(3)of(iii)}
Assume the decomposition matrix with respect to $\cE(\tG,\ell')$ is unitriangular.
Then $\tOmega_{\tcB}$ is $\IBr(\tG/G)$-equivariant.
\end{lem}

\begin{proof}
First, by \cite[Lemma 2.3]{De17} and the assumption on the decomposition matrix, it suffices to prove the same statement for the irreducible ordinary characters corresponding to irreducible Brauer characters in $\IBr(\tB\mid\tR)$.
Recall that the action of $\Irr(\tG/G)$ on $\Irr(\tG)$ is described in Remark \ref{rem:action-tz}.

We use the construction of weights in Proposition \ref{prop:weights} to consider the action of $\IBr(\tG/G) = \Irr(\tG/G)_{\ell'}$ on weights.
First note that $\tC=\group{\tau,C}$ and $\tN=\group{\tau,N}$, thus $\tC/C \cong \tN/N \cong \tG/G$, so we can identify $\Irr(\tG/G)$ with $\Irr(\tC/C)$ and $\Irr(\tN/N)$ by restriction.
To simplify the notation, we abbreviate $\Res^{\tG}_{\tN}$ and $\Res^{\tG}_{\tC}$.

Let $\ts,\tz_0,\ts_0$ be as in Remark \ref{rem:ts0}.
It suffices to consider the action of $\widehat{\tz_0}$.

Assume we are in case (1) in the proof of Proposition \ref{prop:weights}, then the weight $(\tR,\tvarphi)$ has the label of the form $(\ts\tz_0^i,\set{\kappa},K)$.
We may assume $0\leq i\leq(q-1)_{\ell'}/2-1$ for cases (II) and (III).
In these cases, $\tN(\theta_+)=\tN(\ttheta)$.
By the proof of Proposition \ref{prop:weights}, the weight character $\tvarphi$ is of the form $\tvarphi = \Ind_{\tN(\ttheta)}^{\tN} \ttheta_0\psi_{+,K}$.
Then $\widehat{\tz_0}\tvarphi = \Ind_{\tN(\theta_+)}^{\tN} (\widehat{\tz_0}\ttheta_0)\psi_{+,K}$.
By Remark \ref{rem:ts0}, $\ttheta_0$ induces a character $\chi_{\ts_0\tz_0^i,\kappa}$ of $\J_0(V_0)$.
By Remark \ref{rem:zeta}, $\widehat{\tz_0}\ttheta_0$ is determined by $\widehat{\tz_0}\chi_{\ts_0\tz_0^i,\kappa}$.
Thus by Remark \ref{rem:action-tz}, $\widehat{\tz_0}\tvarphi$ is the weight character associated with $\tR$ with label
\begin{enumerate}[(a)]\setlength{\itemsep}{0pt}
\item $(\ts\tz_0^{i+1},\set{\kappa},K)$ for case (I), case (II) or (III) and $i<(q-1)_{\ell'}/2-1$;
\item $(\ts,\set{\kappa},K)$ for case (II) and $i=(q-1)_{\ell'}/2-1$;
\item $(\ts,\set{\kappa'},K)$ for case (III) and $i=(q-1)_{\ell'}/2-1$.
\end{enumerate}

Assume we are in case (2) in the proof of Proposition \ref{prop:weights} i.e. it is the case (V), then the weight $(\tR,\tvarphi)$ has the label of the form $(\ts\tz_0^i,\set{\kappa},K)$.
We may assume $0\leq i\leq(q-1)_{\ell'}/2-1$.
In this case, $\tN(\theta_+)=\tN(\ttheta)$.
By the proof of Proposition \ref{prop:weights}, the weight character $\tvarphi$ is of the form $\tvarphi = \Ind_{\tN(\ttheta)}^{\tN} \psi_K$ with $\psi_K = \Ind_{\hN(\ttheta)}^{\tN(\ttheta)} \htheta_0^{(i)}\psi_{+,K}$, $i=0$ or $1$.
Thus $\widehat{\tz_0}\tvarphi = \Ind_{\tN(\ttheta)}^{\tN} \widehat{\tz_0}\psi_K$ and $\widehat{\tz_0}\psi_K$ is determined by $\widehat{\tz_0}\ttheta_0$ by construction.
Then by the same arguments as in the above paragraph, $\widehat{\tz_0}\tvarphi$ is the weight character associated with $\tR$ with label $(\ts\tz_0^{i+1},\set{\kappa},K)$ if $i<(q-1)_{\ell'}/2-1$ and $(\ts,\set{\kappa},K)$ if $i=(q-1)_{\ell'}/2-1$.

It remains to consider the case (3) in the proof of Proposition \ref{prop:weights}, i.e. (IV) or (VI), in which case, $\tN(\theta_+)\neq\tN(\ttheta)$.

Assume we are in case (IV), then the weight $(\tR,\tvarphi)$ has the label of the form $(\ts\tz_0^i,\set{\kappa},\set{K,K'})$ or $(\ts\tz_0^i,\set{\kappa},(K,j))$ with $j\in\Z/2\Z$.
Assume first the label is $(\ts\tz_0^i,\set{\kappa},\set{K,K'})$ with $K \neq K'$.
We may assume $0\leq i\leq(q-1)_{\ell'}/2-1$.
By the proof of Proposition \ref{prop:weights}, the weight character $\tvarphi$ is of the form $\tvarphi = \Ind_{\tN(\ttheta)}^{\tN} \ttheta_0\psi_{+,K}^0$.
Then $\widehat{\tz_0}\tvarphi = \Ind_{\tN(\theta_+)}^{\tN} (\widehat{\tz_0}\ttheta_0)\psi_{+,K}^0$.
By Remark \ref{rem:ts0}, $\ttheta_0$ induces a character $\chi_{\ts_0\tz_0^i,\kappa}$ of $\J_0(V_0)$.
By Remark \ref{rem:zeta}, $\widehat{\tz_0}\ttheta_0$ is determined by $\widehat{\tz_0}\chi_{\ts_0\tz_0^i,\kappa}$.
If $i<(q-1)_{\ell'}/2-1$, by Remark \ref{rem:action-tz}, $\widehat{\tz_0}\tvarphi$ is the weight character associated with $\tR$ with label $(\ts\tz_0^{i+1},\set{\kappa},\set{K,K'})$.
If $i=(q-1)_{\ell'}/2-1$, $\widehat{\tz_0}\chi_{\ts_0\tz_0^i,\kappa}=\chi_{-\ts,\kappa}$.
By Proposition \ref{prop:N(theta+):N(ttheta)} and its proof, there is $g_+ \in N_+(\theta_+)-N_+(\ttheta)$ such that $\chi_{-\ts,\kappa} = \chi_{\ts,\kappa}^{g_+}$.
By Remark \ref{rem:cc-cano-chars}, we should always start with $\ttheta_0$ corresponding to $\chi_{\ts,\kappa}$ to label the weights.
Thus $\widehat{\tz_0}\tvarphi = \Ind_{\tN(\theta_+)}^{\tN} \ttheta_0^{g_+}\psi_{+,K}^0 = \Ind_{\tN(\theta_+)}^{\tN} \ttheta_0({^{g_+}\psi_{+,K}^0}) = \Ind_{\tN(\theta_+)}^{\tN} \ttheta_0\psi_{+,K}^0$, which has label $(\ts,\set{\kappa},\set{K,K'})$.
In summary, $\widehat{\tz_0}\tvarphi$ is the weight character associated with $\tR$ with label
\begin{enumerate}[(a)]\setlength{\itemsep}{0pt}
\item $(\ts\tz_0^{i+1},\set{\kappa},\set{K,K'})$ if $i<(q-1)_{\ell'}/2-1$;
\item $(\ts,\set{\kappa},\set{K,K'})$ if $i=(q-1)_{\ell'}/2-1$.
\end{enumerate}
Assume then the label is $(\ts\tz_0^i,\set{\kappa},(K,j))$ with $K=K'$ and $j\in\Z/2\Z$.
We may assume $0\leq i\leq(q-1)_{\ell'}/2-1$.
By the proof of Proposition \ref{prop:weights}, the weight character $\tvarphi$ is of the form $\tvarphi = \Ind_{\tN(\ttheta)}^{\tN} \ttheta_0\psi_{+,K}^{(j)}$.
Then $\widehat{\tz_0}\tvarphi = \Ind_{\tN(\theta_+)}^{\tN} (\widehat{\tz_0}\ttheta_0)\psi_{+,K}^{(j)}$.
Thus the result can be obtained similarly as above, noting that ${^{g_+}\psi_{+,K}^{(j)}}=\psi_{+,K}^{(j+1)}$.
In summary, $\widehat{\tz_0}\tvarphi$ is the weight character associated with $\tR$ with label
\begin{enumerate}[(a)]\setlength{\itemsep}{0pt}
\item $(\ts\tz_0^{i+1},\set{\kappa},(K,j))$ if $i<(q-1)_{\ell'}/2-1$;
\item $(\ts,\set{\kappa},(K,j+1))$ if $i=(q-1)_{\ell'}/2-1$.
\end{enumerate}

The result for case (VI) is similar with $\set{\kappa}$ replaced by $\set{\kappa,\kappa'}$ and can be proved in the same way, noting that when $i=(q-1)_{\ell'}/2-1$, $\chi_{-\ts_0,\kappa}$ should be replaced by $\chi_{\ts_0,\kappa'}$ for case (VI).

Then it is easy to see that the bijection in Theorem \ref{thm:bijection} is equivariant under the action of $\Irr(\tG/G)_{\ell'}$, which proves this lemma.
\end{proof}

\subsection{}
In this subsection, we finish the proof of Theorem \ref{mainthm-2}.
Since $\Out(G)$ is abelian, it suffices to verify parts (iii) and (iv) of Theorem \ref{thm:criterion}, and in fact, the requirements in parts (iii) and (iv) are satisfied for any $\chi\in\IBr(G)$ and any $\psi\in\dz(N_G(R)/R)$ respectively.
This would complete the proof of  Theorem \ref{mainthm-2}.

\begin{lem}
Assume the decomposition matrix with respect to $\cE(G,\ell')$ is unitriangular.
Then for any $\chi\in\IBr(G)$,
\begin{enumerate}[(1)]\setlength{\itemsep}{-2pt}
\item $(\tG\rtimes E)_\chi=\tG_\chi \rtimes E_\chi$,
\item $\chi$ extends to $G\rtimes E_\chi$.
\end{enumerate}
\end{lem}

\begin{proof}
By \cite[Theorem 3.1]{CS17C}, $(\tG\rtimes E)_\chi=\tG_\chi\rtimes E_\chi$ holds for any $\chi\in\Irr(G)$.
By \cite{Ge93}, $\cE(G,\ell')$ is a basic set of $\Irr(G)$.
Then (1) follows from the assumption about the decomposition matrix with respect to $\cE(G,\ell')$ and \cite[Lemma 7.5]{CS13}.
Since $E$ is cyclic, (2) obviously holds.
\end{proof}

\begin{lem}
Let $R = R_0 \times R_1 \times \cdots \times R_u$ be a radical subgroup of $G$, where $R_0$ is the trivial group and $R_i=R_{m_i,\alpha_i,\gamma_i,\bc_i}$ ($i\geq1$) is a basic subgroup.
For any $\psi\in\dz(N_G(R)/R)$, there exists some $x \in \tG$ with
\begin{enumerate}[(1)]\setlength{\itemsep}{-2pt}
\item $(\tG E)_{R,\psi}= \tG_{R,\psi} (GE)_{R,\psi}$,
\item $\psi$ extends to $(G\rtimes E)_{R,\psi}$.
\end{enumerate}
\end{lem}

\begin{proof}
To prove this lemma, we transfer to the twisted groups as in \cite[Proposition 5.3]{CS17}.
Let $v,g,\iota$ be as in \cite[Lemma 6.7]{Li19}.
Let $\tG^{tw}=\tbG^{vF}$.
Then $\iota$ can be extended to a surjective homomorphism
\[\iota:\quad \tG^{tw} \rtimes \hE \to \tG \rtimes E\]
where $\hE=\group{\hF_p}$ is the group of field automorphisms of $\tG^{tw}$.
Note that $\Ker\iota = \group{v\hF_q}$, where $\hF_q=\hF_p^f$.
Let $R^{tw} = \iota^{-1} (R)$ and $\psi^{tw} = \psi\circ\iota$.
By \cite[Lemma 6.8]{Li19}, $\sigma(R^{tw})=R^{tw}$ for any $\sigma\in\hE$ and $N_{\tG^{tw}\hE}(R^{tw}) = N_{\tG^{tw}}(R^{tw}) \rtimes \hE$.
Then by a similar argument as in \cite[Proposition 5.3]{CS17}, it suffices to prove the following
\begin{enumerate}\setlength{\itemsep}{0pt}
\item[(1')] $(N_{\tG^{tw}}(R^{tw}) \rtimes \hE)_{\psi^{tw}} = N_{\tG^{tw}}(R^{tw})_{\psi^{tw}} \rtimes \hE_{\psi^{tw}}$,
\item[(2')] $\psi^{tw}$ can be extended to a character $\tpsi^{tw}$ of $N_{G^{tw}}(R^{tw}) \rtimes \hE_{\psi^{tw}}$ with $v\hF_q \in \Ker\tpsi^{tw}$.
\end{enumerate}
(1') follows from \cite[Proposition 6.20, Proposition 6.25]{Li19} and their proofs.
For (2'), note that $v\hF_q$ acts trivially on $G^{tw}$.
Thus we can view $\psi^{tw}$ as a character of $N_{G^{tw}}(R^{tw}) \rtimes \group{\hF_q} = N_{G^{tw}}(R^{tw}) \times \group{v\hF_q}$ containing $v\hF_q$ in the kernel.
Now, since $N_{G^{tw}}(R^{tw}) \rtimes \hE_{\psi^{tw}} / N_{G^{tw}}(R^{tw}) \rtimes \group{\hF_q}$ is cyclic, (2') follows obviously.
Of course, (2) also follows directly from the fact that $E$ is cyclic.
\end{proof}

When the radical subgroup $R^x$ ($R$ is as in the above lemma) is considered, we need to replace $E$ with $E^x$.

\section{Proof of Theorem \ref{mainthm-3}}\label{sec:main-3}

In this section, we assume $q=2^f$ and consider the simple group $G=\Sp_{2n}(2^f)$ with $(n,f)\neq(2,1)$ or $(3,1)$.
For this case, our proof is just to say that all the relevant arguments in \cite{FS89,An94,Li19} and the previous sections apply.
We list all the statements and point out where to find the relevant proofs.

\begin{lem}
Any radical subgroup of $G$ is conjugate to a subgroup of the form $R_0 \times R_1 \times \cdots \times R_u$, where $R_0$ is a trivial group and $R_i=R_{m_i,\alpha_i,\gamma_i,\bc_i}$ is a basic subgroup for $i>0$.
\end{lem}
\begin{proof}
The relevant arguments in \cite[\S1,\S2]{An94} apply to this case.
\end{proof}

\begin{rem}
In \cite[Proposition 3.1]{SF14}, the author considered $\Sp_6(q)$ and already noted that the relevant arguments in \cite[\S1, \S2]{An94} apply when $\ell \mid (q^2-1)$.
For radical subgroups $Q^{(3)}$ when $3\neq\ell\mid(q^4+q^2+1)$ and $Q^{(2)}$ when $\ell\mid(q^2+1)$ in \cite[Proposition 3.1]{SF14},
we can find the corresponding notation used in \S\ref{subsec:radical}.
Note that $e$ is necessarily odd if $\ell$ is linear.
When $3\neq\ell\mid(q^4+q^2+1)$, $e=3$, and $\ell$ is linear if $\ell\mid(q^2+q+1)$ while $\ell$ is unitary if $\ell\mid(q^2-q+1)$.
In this case, $Q^{(3)}$ is just $R_1$ defined in \S\ref{subsec:radical}; here, recall that $R_1$ is the abbreviation of $R_{1,0,0,\zero}$.
When $\ell\mid(q^2+1)$, $e=2$ and $\ell$ is unitary, and $Q^{(2)}$ should be again denoted as $R_1$ using the notation system in \S\ref{subsec:radical}.
\end{rem}

Let $D_{m,\alpha,\beta}$ be defined in the same way as in \S\ref{subsec:ano-conju}.
\begin{lem}
Any defect group of $G$ is of the form $D_0 \times D_1 \times \cdots \times D_u$, where $D_0$ is the trivial group and $D_i=D_{m_i,\alpha_i,\beta_i}$ for $i>0$.
\end{lem}
\begin{proof}
The argument in \cite[(5K)]{FS89} applies.
\end{proof}

All the relevant constructions and notation in \S\ref{subsec:Brauerpair} apply to the case in this section.
In particular, let $(R,\varphi)$ be a weight of $G$ and $\varphi$ lies over $\theta\in\dz(C/Z(R))$, then we can start with a Brauer pair $(R,\theta)$ and construct $(D',\theta')$ and $(D,\theta_D)$.

\begin{lem}
$(R,\theta) \leq (D,\theta_D)$ as Brauer pairs in $G$.
$(D,\theta_D)$ is a maximal Brauer pair and all maximal Brauer pair of $G$ are of this form.
\end{lem}
\begin{proof}
The argument of Proposition \ref{prop:Brauerpair} applies.
\end{proof}
With this lemma, we can determine which block the weight $(R,\varphi)$ belongs to as we have done for $q$ odd.

Let $(D,\theta_D)$ be a maximal Brauer pair for the block $B$ of $G$.
Let $(s,\kappa)$ be defined in the same way as in \cite[\S10]{FS89}, then $(s,\kappa)$ is a label of $B$; such labelling gives a bijection as in \cite[(10B)]{FS89}.
The dual defect group of $B$ is defined as the image of $D$ under the isomorphism $G \cong G^*$.

\begin{lem}
Let $B$ be the block of $G$ with label $(s,\kappa)$.
Then for any $\chi\in\Irr(G)$, $\chi\in\Irr(B)$ if and only if $\chi=\chi_{t,\lambda}$ and (1) $t_{\ell'}$ is conjugate to $s$; (2) $t_\ell$ is contained in a dual defect group of $B$; (3) $\kappa$ is the core of $\lambda$.
\end{lem}
\begin{proof}
The arguments of \cite[(12A)]{FS89} apply.
Here, note that $\kappa$ is the core of $\lambda$ means $\kappa_\Gamma$ is the $e_\Gamma$-core of $\lambda_\Gamma$.
\end{proof}

\begin{cor}
$\IBr(B_{s,\kappa})$ can be labelled by $\phi_{s,\lambda}$, where $\kappa$ is the core of $\lambda$.
\end{cor}

\begin{rem}
As an example, see \cite{Whi95} and \cite{Whi00} for irreducible ordinary characters in unipotent blocks of $\Sp_4(2^f)$ and $\Sp_6(2^f)$, and see \cite{SF13} for irreducible ordinary characters in non-unipotent blocks of $\Sp_6(2^f)$.
\end{rem}

The parametrization of weights is the same as in \cite{An94}; see \cite[\S6.B]{Li19} for a description of the construction of weights.
In the above constructions, the duals of semisimple elements are used; see on \cite[p.18]{An94} for the definition of dual for the case $q$ odd.
In the case $q$ even, we define the dual of semisimple elements via the isomorphism $G \cong G^*$.
We state the result as follows.

\begin{lem}[{\cite[4F]{An94}}]
Let $B=B_{s,\kappa}$ be a block of $G$.
Assume $s= s_0 \times s_+$ as in \cite[\S10]{FS89}.
Then $m_\Gamma(s)-m_\Gamma(s_0)=w_\Gamma\beta_\Gamma e_\Gamma$ for some natural number $w_\Gamma$.
Set
\begin{equation*}
i\cW(B)=\left\{
Q=\prod_\Gamma Q_\Gamma ~\middle|~
\begin{array}{c}
Q_\Gamma=\left(Q_\Gamma^{(1)},Q_\Gamma^{(2)},\dots,Q_\Gamma^{(\beta_\Gamma e_\Gamma)}\right),\\
\textrm{$Q_\Gamma^{(j)}$'s are partitions},\sum\limits_{j=1}^{\beta_\Gamma e_\Gamma} |Q_\Gamma^{(j)}|=w_\Gamma.
\end{array}
\right\}
\end{equation*}
Here $Q_\Gamma$ is an ordered sequence of $\beta_\Gamma e_\Gamma$ partitions.
Then there is a bijection between $i\cW(B)$ and $\Alp(B)$.
\end{lem}

The conjugacy class of weights in $B_{s,\kappa}$ corresponding to $Q$ will be denoted as $w_{s,\kappa,Q}$.

\begin{lem}
Let $B=B_{s,\kappa}$ be a block of $G$.
Then there is a bijection
\[\IBr(B) \to \Alp(B), \quad \phi_{s,\lambda} \mapsto w_{s,\kappa,Q},\]
where $\lambda_\Gamma=\kappa_\Gamma*Q_\Gamma$.
\end{lem}
\begin{proof}
Noting that $\lambda_{X-1}$ and $\kappa_{X-1}$ are both non-degenerate Lusztig symbols, while $\lambda_\Gamma$ and $\kappa_\Gamma$ for $\Gamma\neq X-1$ are partitions, the result follows from \S\ref{subsec:partition-symbol}.
\end{proof}

\begin{lem}
Assume the sub-matrix of the decomposition matrix of $\Sp_{2n}(2^f)$ to $\cE(\Sp_{2n}(2^f),\ell')$ is unitriangular.
Then the above bijection is $\Aut(G)$-equivariant.
\end{lem}
\begin{proof}
Since the cases $n=1$ and $n=2$ have been considered in \cite{Sp13} and \cite{SF14} respectively,
we may assume $n\geq3$, thus it suffices to consider $\sigma\in E$.
For $\chi_{s,\lambda}\in\Irr(B)$, $\chi_{s,\lambda}^\sigma = \chi_{\sigma^*(s),\sigma^*(\lambda)}$ by \cite[Theorem 3.1]{CS13} and similar arguments in \cite[Proposition 5.6]{Li19}.
Then by the assumption on the decomposition matrix, $\phi_{s,\lambda}^\sigma = \phi_{\sigma^*(s),\sigma^*(\lambda)}$.
For the action of $\sigma$ on weights, the arguments in \cite[\S6.C]{Li19} applies.
Then it is easy to see the above bijection is equivariant.
\end{proof}

Let $\tG=G$ and $\cB=\set{B}$, then the other requirements in Theorem \ref{thm:criterion} are obviously satisfied.
This completes the proof of Theorem \ref{mainthm-3}.

\section*{Acknowledgements}
The author is extremely grateful to Julian Brough, Marc Cabanes and Britta Sp\"ath for their hospitality during his visit at Wuppertal in May 2019, and especially for the fruitful discussions and keen suggestions, which made this paper possible.
The author acknowledges the great working atmosphere during the stay of a month in the summer of 2019 at Sustech International Center for Mathematics, where the main part of this paper was written.

\end{document}